\documentclass[preprint,11pt]{article}

\usepackage{amssymb,amsfonts,amsmath,amsthm,amscd,dsfont,mathrsfs}
\usepackage{graphicx,float,psfrag,epsfig}
\usepackage{wrapfig}
\usepackage{relsize}
\usepackage{color}
\usepackage{pict2e}
\usepackage[tight]{subfigure}
\usepackage{algorithm}
\usepackage[noend]{algorithmic}
\usepackage{caption}

\DeclareMathAlphabet{\mathpzc}{OT1}{pzc}{m}{it}

\newtheorem{propo}{Proposition}[section]
\newtheorem{lemma}[propo]{Lemma}
\newtheorem{definition}[propo]{Definition}
\newtheorem{coro}[propo]{Corollary}
\newtheorem{thm}[propo]{Theorem}

\def\test{{\xi}}
\def \sB{{\sf B}}

\def \sH{{\sf H}}
\def \sE{{\sf E}}
\def\vz{{\mathbf z}}

\def\cH{{\cal H}}

\def\cE{{\cal E}}

\def\vt{{\vartheta}}

\def\naturals{{\mathbb N}}

\def\reals{{\mathbb R}}

\def\Prox{{\sf Prox}}
\def\dual{\Psi}

\def\prob{{\mathbb P}}
\def\E{{\mathbb E}}
\def\Var{{\rm Var}}
\def\MSE{{\sf MSE}}

\def\L0{{L_0}}

\def\de{{\rm d}}
\def\<{\langle}
\def\>{\rangle}

\def\bZ{{\mathbf Z}}
\def\bX{{\mathbf X}}

\def\bV{{\mathbf V}}
\def\tbX{\widetilde{\mathbf X}}
\def\tn{\tilde{n}}
\def\tp{\tilde{p}}
\def\tY{\widetilde{Y}}
\def\hbeta{\widehat{\beta}}

\def\htheta{\widehat{\theta}}

\def\F{{\sf F}}

\def\F{{\sf F}}
\def\normal{{\sf N}}
\def\Cnormal{{\sf CN}}

\def\sT{{\sf T}}

\def\id{{\rm I}}

\def\sign{{\rm sign}}

\def\v*{v_0}
\def\T*{T_0}

\def\u*{u_0}
\def\F*{F_0}

\definecolor{olivegreen}{rgb}{0,0.6,0.4}

\def\cL{{\cal L}}

\def\Corr{{\Gamma}}
\def\Huber{\rho_{{\rm H}}}

\def\MSE{{\rm MSE}}
\def\MAE{{\rm MAE}}
\def\RMSE{{\rm RMSE}}
\def\AMSE{{\rm AMSE}}
\def\cV{{\cal V}}

\newcommand{\bitem}{\begin{itemize}}
\newcommand{\eitem}{\end{itemize}}
\newcommand{\goto}{\to}
\newcommand{\beq}{\begin{equation}}
\newcommand{\eeq}{\end{equation}}

\newcommand{\ajcomment}[1]{}

\makeatletter
\newcommand{\labitem}[2]{%
\def\@itemlabel{\text{#1}}
\item
\def\@currentlabel{#1}\label{#2}}

\usepackage{bibentry}
\newcommand{\ignore}[1]{}
\newcommand{\nobibentry}[1]{{\let\nocite\ignore\bibentry{#1}}}

\makeatother

\addtocontents{toc}{\protect\setcounter{tocdepth}{2}}

\title{High Dimensional Robust M-Estimation:\\
 Asymptotic Variance via Approximate Message Passing}

\author{David Donoho
            \footnote{Department of Statistics, Stanford
              University} \, and \,Andrea~Montanari 
            \footnote{Department of Electrical Engineering and
              Department of Statistics, Stanford University}
            }

\begin{document}

\maketitle

\begin{abstract}
In a recent article (\emph{Proc. Natl. Acad. Sci.}, 110(36),
14557-14562), El Karoui et al. study  the distribution of robust regression
estimators in the regime in which the number of parameters $p$ 
is of the same order as the number of samples $n$. Using numerical 
simulations and `highly plausible' heuristic arguments, they 
unveil a striking new phenomenon. Namely, the regression coefficients 
contain an extra Gaussian noise component that is not explained
by classical concepts such as the Fisher information matrix.

We show here that that this phenomenon can be characterized rigorously
using techniques developed by the authors to analyze the
Lasso estimator under high-dimensional asymptotics. We introduce
 an \emph{approximate message passing} (AMP) algorithm to
compute  M-estimators and deploy \emph{state evolution} to 
evaluate the operating characteristics of AMP and so also M-estimates.
 Our analysis clarifies that the
`extra Gaussian noise' encountered in this problem is fundamentally similar to 
phenomena already studied for regularized least squares in the setting $n<p$.
\end{abstract}

\section{M-Estimation under high dimensional asymptotics}

Consider the traditional  linear regression model
\begin{eqnarray}\label{eq:NoisyModel}
Y\, =\, \bX\,\theta_0+ W\, ,\label{eq:Model}
\end{eqnarray}
with  $Y = (Y_1,\dotsc,Y_n)^\sT\in\reals^n$ a vector of
responses, $\bX\in \reals^{n\times p}$  a known design matrix,
$\theta_0\in\reals^p$ a vector of parameters, and $W\in\reals^n$ 
random
noise having zero-mean components
$W = (W_1,\dotsc,W_n)^{\sT}$ 
i.i.d.  with distribution $F = F_{W}$ 
having  finite second 
moment \footnote{With a slight abuse of notation, we
shall use $W$ to denote a random variable with the same distribution
$F_W$.}.

We are interested in estimating $\theta_0$ from observed data\footnote{We  denote by $X_1$, \dots, 
$X_n$ the rows of $\bX$.
 We often omit
the arguments $Y$, $\bX$ as this dependency will hold throughout.
Without loss of generality, we assume that the columns of 
$\bX$ are normalized so that $\|\bX\,e_i\|_2\approx 1$. (A more
precise assumption will be formulated below.) }
$(Y,\bX)$ using a traditional M-estimator, defined by a non-negative
convex function $\rho : \reals \to \reals_{\ge 0}$:
\begin{align}
\htheta (Y;\bX)  \equiv \arg\min_{\theta\in\reals^p}\cL(\theta;Y,\bX)\,
,\;\;\;\;\;\;\;\;\;
\cL(\theta;Y,\bX) \equiv \sum_{i=1}^n
\rho\big(Y_i-\<X_i,\theta\>\big)\, ,
\label{eq:Mestimation}
\end{align}
where $\<u,v\>=\sum_{i=1}^mu_iv_i$ is the standard scalar product in
$\reals^m$, and $\htheta$ is chosen arbitrarily if there is multiple  minimizers.

Although this is a completely traditional
problem, we consider it under \emph{high-dimensional asymptotics}
where the number of parameters $p$ and the number of observations $n$
are both tending to infinity, at the same rate.
This is becoming a popular asymptotic model
owing to the modern awareness of `big data'
and `data deluge'; but also because it
leads to entirely new phenomena.

\subsection{Extra Gaussian noise due to high-dimensional asymptotics}

Classical statistical theory considered the situation
where the number of regression parameters $p$ is fixed
and the number of samples $n$ is tending to infinity.
The asymptotic distribution was found by Huber \cite{huber1973robust,bickel1975} to be normal $\normal(0,\bV)$
where the asymptotic variance matrix $\bV$ is given by
\begin{equation}
 \bV = V(\psi,F_W) (\bX^{\sT}\bX)^{-1}   \label{eq:CLASSICALVAR}
\end{equation}
here $\psi = \rho'$ is the score function of the M-estimator
and $V(\psi,F) = (\int \psi^2 \de F)/(\int \psi' \de F)^2$ the asymptotic variance
 functional of \cite{HuberMinimax}, and $(\bX^{\sT}\bX)$ the usual Gram matrix
 associated with the least-squares problem. Importantly, it was found that
 for efficient estimation -- i.e. the smallest possible asymptotic variance --
 the optimal M-estimator depended on the probability distribution $F_W$ of the errors $W$.
Choosing   $\psi(x)=  (\log f_W(x))'$ (with $f_W$ the density of $W$), the asymptotic variance
functional yields $V(\psi,F_W) = 1/I(F_W)$, with $I(F)$ denoting
the Fisher information. 
This  achieves the fundamental limit on the
accuracy of M-estimators \cite{huber1973robust}.

In modern statistical practice there is increasing interest
in applications where the number of explanatory variables
$p$ is very large, and comparable to $n$.
Examples of this new regime can be given,
spanning bioinformatics,  machine learning,
imaging, and signal processing 
(a few research areas in the last domains include \cite{CSMRI,Seismic,Radar,Hyperspectral}).

This paper considers the properties of M-estimators
in the high-dimensional asymptotic $n \to \infty$,
$n/p(n) \to \delta \in (1,\infty)$  In this regime,
the asymptotic distribution of M-estimators
no longer needs to obey the classical formula (\ref{eq:CLASSICALVAR})
in widespread use. 
We make a random-design assumption on the $\bX$'s 
detailed below.
We show that the asymptotic covariance matrix
of the parameters is now of the form
\begin{equation} \label{HDAVarCov}
      \bV = V(\tilde{\Psi},\tilde{F}_W) (\E\{\bX^{\sT}\bX\})^{-1},
\end{equation}
where $V$ is still Huber's asymptotic variance functional,
but $\tilde{\Psi}$ is the  \emph{effective score function},
which is different from $\psi$ under high-dimensional asymptotics
and $\tilde{F}_W$ is the \emph{effective error distribution},
which is different from $F_W$ under high-dimensional asymptotics.
In the limit $\delta \to \infty$,  the effective score and the effective
error distribution both tend to their classical counterparts, and
one recovers $V(\psi,F_W)$.

The effective error distribution $\tilde{F}_W$
is a convolution of the noise distribution with an extra Gaussian
noise component, not seen in the classical setting (here $\star$
denotes convolution):
\begin{equation} \label{EffectiveNoise}
    \widetilde{F}_W\equiv F_W \star \normal(0,\tau_*^2(\psi,F_W,\delta))\, .
\end{equation}
The extra Gaussian noise 
depends in a complex way on $\psi$, $F_W$, $\delta$, which we characterize fully below
in Corollary \ref{coro:AysVariance}.

Several important insights follow immediately:
\begin{enumerate}
\item Existing formulas are inadequate for confidence
statements about M-estimates under high dimensional asymptotics,
and will need to be systematically broadened.
\item Classical maximum likelihood estimates are
inefficient under  high-dimensional asymptotics.
The idea dominating theoretical statistics since R.A. Fisher to use
$\psi = (- \log f_W)'$ as a scoring rule,
does not yield the efficient estimator.
\item The usual Fisher Information bound is not necessarily attainable
in the high-dimensional asymptotic, as $I(\widetilde{F}_W) < I(F_W)$.
\end{enumerate}

M-estimation in this high-dimensional
asymptotic setting was considered in a recent article by
El Karoui, Bean, Bickel, Lim, and Yu \cite{karoui2013robust}, who
studied the distribution of $\htheta$ for Gaussian
design matrices $\bX$. 
In short they observed empirically the basic phenomenon of
extra Gaussian noise appearing in high-dimensional
asymptotics and rendering classical inference incorrect.
The dependence of the additional
variance $\tau_*^2$ on  $\delta$, $\psi$ and $F$ was characterized by \cite{karoui2013robust}
through a non-rigorous heuristics \footnote{To the reader familiar with the mathematical theory of spin glasses, 
the argument of \cite{karoui2013robust} appears analogous to the cavity method
from statistical physics
\cite{SpinGlass,MezardMontanari,TalagrandVolI}}
that the authors describe as `highly plausible and
buttressed by simulations.'\footnote{After the first version of our manuscript was posted on ArXiv, 
      Noureddine El Karoui announced an
      independent proof of related results, using a completely different
      approach.}
(We refer to Section \ref{sec:Discussion} for further discussion of
related work.)

\subsection{Proof Strategy: Approximate Message Passing }

In the present paper, we show that this important statistical phenomenon can be
\emph{characterized rigorously}, in a way that we think fully explains the main new concepts
of extra Gaussian noise,
effective noise and the effective score. Our proof  strategy has three steps
\bitem
 \item Introduce an {\it Approximate Message Passing} (AMP) algorithm for M-estimation;
an iterative procedure with the M-estimator as a fixed point, and having the effective score
function $\tilde{\Psi}$ as its score function at algorithm convergence.
 \item Introduce {\it State Evolution} for calculating properties of the
 AMP algorithm iteration by iteration. We show that these calculations are exact at each iteration
  in the large-$n$ limit where we freeze the iteration number and let $n \to \infty$.
  
  At the center of the State Evolution calculation is  precisely an extra Gaussian noise term
  that is tracked from iteration to iteration, and which is shown to converge to a nonzero
  noise level. In this way, State Evolution makes very explicit that AMP faces at each iteration
  and even in the limit, an effective noise that differs from the noise $W$ by addition
  of an appreciable extra independent Gaussian noise.
  
 \item Show  that the AMP algorithm converges to the solution of
 the M-estimation problem in mean square, from which it follows that
 the asymptotic variance of the M-estimator is identical to
 the asymptotic variance of  the AMP algorithm. More specifically,
 the asymptotic variance of the M-estimator is given by a formula
 involving the effective score function and the effective noise. 
 \eitem

As it turns out,  our formula for the asymptotic variance
coincides with the one derived heuristically in
\cite[Corollary 1]{karoui2013robust} although our technique is
remarkably different, and our proof provides a very clear 
understanding of the operational significance of the terms
appearing in the asymptotic variance.  It also allows
explicit calculation of many other operating characteristics of the
M-estimator, for example when used as an outlier detector\footnote{The slightly more general \cite[Result
  1]{karoui2013robust} covers heteroscedastic noise is not covered by 
  the analysis of this paper, but should be provable
  by adapting our argument.}.

\subsection{Underlying tools}

At the heart of our analysis, we are simply applying an approach developed in
\cite{BM-MPCS-2011,BayatiMontanariLASSO} for rigorous analysis of 
solutions to convex optimization problems
under high-dimensional asymptotics.

That approach grew out of a series of earlier papers studying the
compressed sensing problem \cite{DMM09,DMM-NSPT-11,donoho2011compressed,BayatiMontanariLASSO}.
From the perspective of this paper, those papers considered the 
same regression model (\ref{eq:Model}) as here;  
however, they emphasized the challenging asymptotic
regime where there are fewer observations than predictors, (i.e.  $n/p(n)\to\delta\in (0,1)$)
so that even in the noiseless case, the equations $Y = \bX\theta$ would be underdetermined.
In the $p > n$ setting, it became popular to
use $\ell_1$-penalized least squares (Lasso, \cite{Tibs96,BP95}).
That series of papers  considered the Lasso convex optimization problem 
in the case of $\bX$ with iid $\normal(0,1/n)$ entries (just as here)
and followed the same 3-step strategy
we use here; namely, 1. Introducing an   
AMP algorithm; 2. Obtaining the asymptotic distribution of AMP
by State Evolution; and 3. Showing that AMP agrees with the Lasso solution in the large-$n$ limit.
This procedure proved that the Lasso solution has the asymptotic distribution
\begin{align}
\htheta^u\sim\normal(\theta_0, (\sigma^2+\tau_{\rm Lasso}^2)\id_{p\times p})
\end{align}
where $\sigma^2$ is the variance of the noise in the measurements,
and $\tau_{\rm Lasso}^2$ is the variance
of an extra Gaussian noise, not appearing in the classical setting where $p(n)/n \goto 0$.
The variance of this extra Gaussian noise was obtained by state evolution
and shown to depend on the distribution of the coefficients
being recovered, and on the noise level in a seemingly
complicated way that can be characterized
by a fixed-point relation, see \cite{DMM-NSPT-11,BayatiMontanariLASSO}.
At the center of the rigorous analysis stand the papers  
\cite{BM-MPCS-2011,BayatiMontanariLASSO} 
which analyze recurrences of the
type used by AMP and establish the validity of State Evolution
in considerable generality. Those same papers
stand at the center of our analysis in this paper. 

Apart from allowing a
simple treatment, this provides a unified understanding of  the
phenomenon of 
high-dimensional extra Gaussian noise.

 \subsection{The role of AMP}

This paper introduces a new 
first-order algorithm for computing the M-estimator $\htheta$
which is uniquely appropriate for the random-design case. 
This algorithm fits within the class of approximate message
passing (AMP) algorithms introduced in \cite{DMM09,BM-MPCS-2011}
(see also \cite{RanganGAMP} for extensions). 
This algorithm is of independent interest because of its low
computational complexity. 

AMP has a deceptive simplicity. 
As an iterative procedure for convex optimization,
it looks almost the same as the `standard' 
application of simple fixed-stepsize gradent descent.
However, it is intended for use in the random-design setting,
and it has an extra memory term (aka reaction term)
that modifies the iteration in a profound and beneficial way.
In the Lasso setting, AMP algorithms have 
been shown to have remarkable fast convergence properties \cite{DMM09},
far outperforming more complex-looking iterations like Nesterov
and FISTA.

In the present paper, AMP has an second important wrinkle --
it solves a convex optimization problem associated to minimizing
$\rho$ with iterations based on gradient descent with an objective $\rho_{b_t}$
which varies from one iteration to the next, as $b_t$ changes,
but which does not tend to $\rho$ in the limit.

In the present paper, AMP is mainly used as a proof device,
one component of the three-part strategy outlined earlier.   However,
a key benefit produced by the curious features of AMP is
strong heuristic insight, which would not be available
for a `standard' gradient-descent algorithm.

The AMP proof strategy makes visible the extra Gaussian noise appearing in
the M-estimator $\htheta$. 
Elementary considerations show that such extra noise
is present at iteration zero of AMP.
State Evolution faithfully tracks the dynamics of this
extra noise across iterations. State Evolution
proves that the extra noise level does not go to zero
asymptotically with increasing iterations, but instead
that the extra noise level tends to a fixed nonzero value.
Because AMP is solving the M-estimation
problem, the M-estimator must be infected by this
extra noise.

The AMP algorithm and its State Evolution analysis shows that 
the extra noise in parameter
$\htheta_i^t$ at iteration $t$ is due to cross-parameter 
estimation noise leakage, where 
errors in the estimation of all other parameters at the
previous iteration $(t-1)$ cause extra noise to appear in $\htheta_i^t$.
In the classical setting no such effect is visible.
One could say that the central fact about the high-dimensional
 setting revealed here as well as in our earlier work
\cite{DMM09,DMM-NSPT-11,donoho2011compressed,BayatiMontanariLASSO},
is that when there are so many parameters
to estimate, one cannot really insulate the estimation
of any one parameter from the errors in estimation of all the other 
parameters.

%
%
\section{Approximate Message Passing (AMP)}

\subsection{A family of score functions}

For the rest of the paper, we make the following smoothness assumption  on $\rho$:
\begin{definition}
We call the loss function $\rho:\reals\to\reals$ \emph{smooth} if it is  continuously 
  differentiable, with absolutely continuous derivative $\psi = \rho'$ having an a.e. derivative  $\psi'$ that
  is bounded: $\sup_{u\in \reals}\psi'(u)<\infty$. 
\end{definition}
Our assumption excludes some interesting cases, such as
$\rho(u) = |u|$, but includes for instance the Huber loss
\footnote{We expect that the proof technique developed in
this paper should be generalizable to a broader class of functions
$\rho$, at the cost of additional technical complications.}
\begin{align}
\Huber(z; \lambda) = \begin{cases} 
z^2/2 & \mbox{if $|z|\le \lambda$,}\\
\lambda |z| - \lambda^2/2 & \mbox{otherwise.}
\end{cases}\label{eq:HuberDef}
\end{align}

Associated to $\rho$,  we introduce the family $\rho_b$ of regularizations of $\rho$:
\begin{align}
 \rho_b(z) \equiv\min_{x\in\reals} \Big\{b\rho(x) +
 \frac{1}{2}(x-z)^2\Big\}\, ,
\end{align}
 in words,  this is the min-convolution of the original loss with a square loss.
Each $\rho_b$ has a corresponding
 score function
 \[
 \dual(z;b) = \rho'_b(z).
 \]
  The effective score
 of the M-estimator belongs to this family, for a particular choice of $b$,
 explained below.
 
In the classical  M-estimation literature \cite{HuberBook},
monotonicity and differentiability of the
score function $\psi$ is frequently useful; our assumptions on
$\rho$ guarantee these properties for the nominal score function $\psi$.
 The score family $\dual(\,\cdot\, ; b)$ has such properties as well:
 for any $b$, 
 $\dual(\,\cdot\,;b)$ is a strictly monotone increasing  
 function; second,  for any $b>0$,  $\dual(\,\cdot\,;b)$ 
 is a contraction. With $\dual'$
 denoting differentiation with respect to the first variable,
 we have $\dual'(z;b) \in (0,1)$.
For proof and further discussion, see Appendix \ref{app:Properties}.

Before proceeding, we give an example.
Consider the Huber loss
$\Huber(z; \lambda)$,
with score function $\psi(z;\lambda) = \min(\max(-\lambda,z),\lambda)$.
We have
\[
   \dual( z ; b) =  b \psi\Big( \frac{z}{1+b} ; \lambda\Big) .
\]
In particular the shape of each $\Psi$ is similar to $\psi$,
but the slope of the central part is now  $\|\Psi'( \,\cdot\,;b) \|_\infty = \frac{b}{1+b} < 1$.


\subsection{AMP algorithm}

\newcommand{\res}{R}
Our proposed approximate message passing (AMP) algorithm
for the optimization problem (\ref{eq:Mestimation}) 
 is iterative, starting at iteration $0$ with an initial estimate
 $\htheta^0\in\reals^p$. At iteration  $t=0,1,2,\dots$ it 
 applies a simple procedure to update its estimate
 $\htheta^{t}\in\reals^p$, producing $\htheta^{t+1}$.   
 The procedure involves three steps.
 \begin{description}
 \item [Adjusted residuals.]
 Using the current estimate $\htheta^t$,
 we  compute the vector of {\it adjusted 
residuals} $\res ^t\in\reals^n$, 
\begin{align}
\res^t & = Y -\bX\htheta^t+\dual(\res^{t-1};b_{t-1})\, ; \label{eq:AMP1}
\end{align}
where to the ordinary residuals $Y - \bX \htheta^t$ we here add the extra term\footnote{Here and below, given $f:\reals\to\reals$
and $v=(v_1,\dots,v_m)^{\sT}\in\reals^m$, we define $f(v)\in\reals^m$
by applying $f$ coordinate-wise to $v$, i.e. $f(v) \equiv
(f(v_1),\dots ,f(v_m))^{\sT}$.}
$\dual(\res^{t-1};b_{t-1})$.

\item [Effective Score.]
We choose a  scalar $b_t > 0$,
so that the effective score $\dual(\,\cdot\,;  b_t)$ has empirical average slope $p/n \in (0,1)$.
Setting $\delta=\delta(n) = n/p>1$, we take any solution\footnote{This equation always admits at
  least one solution since $b\mapsto \dual'(r;b)$ is continuous in
  $b\ge 0$, with $\dual'(r;0)=0$ and (for $\rho$ strictly convex)
  $\dual'(r;\infty)=1$, cf. Proposition \ref{propo:Prox}.} (for instance the
smallest solution) to \footnote{Under this prescription,
the sequence $b_t$ depends on the instance $(Y,\bX)$. As explained in
the next section, for the proof of our main result we will use a
slightly different prescription, that is independent of the problem instance.}:
\begin{align}
 \frac{1}{\delta} = \frac{1}{n}\sum_{i=1}^n\dual'(\res^t_i;b)\, .  \label{eq:AMPb}
\end{align}
\item [Scoring.] We apply the effective score function $\dual(\res^t; b_t)$:
\begin{align}
%
\htheta^{t+1} & = \htheta^t+\delta\bX^{\sT}\dual(\res^t;b_t)\, . \label{eq:AMP2}
\end{align}

\end{description}


%

The Scoring step of the AMP iteration  (\ref{eq:AMP2}) is
 similar to traditional iterative methods for M-estimation, compare
\cite{bickel1975}.  Indeed, using the traditional
residual $z^t = Y - \bX \theta^t$, the traditional method of scoring at iteration $t$ 
would read
\beq \label{eq:TradMethScore}
       \htheta^{t+1}  = \htheta^t +   \frac{1}{\frac{1}{n}\sum_{i=1}^n\psi'(z^t_i)} (\bX^{\sT}\bX)^{-1}  \bX^{\sT} \psi(z^t),
\eeq
and one can see  correspondences of individual terms to the method of scoring used in AMP.
Of course the traditional term $ [\sum_{i=1}^n\psi'(z^t_i)/n]^{-1}$
corresponds to  AMP's $ [\sum_{i=1}^n\dual'(\res^t_i; b_t)/n]^{-1} \equiv \delta$ (because of step (\ref{eq:AMPb})),
while the traditional term $(\bX^{\sT}\bX)^{-1}$ corresponds to AMP's implicit $\id_{p\times p}$ -- which is
appropriate in the present context because 
our  random-design assumption below makes $\bX^{\sT}\bX$  
behave approximately like the identity matrix.

\subsection{Relation to M-estimation}

The next lemma explains the reason for using 
the effective score $\dual(\cdot ; b_t)$ in the AMP algorithm: this is what
connects the AMP iteration to  M-estimation (\ref{eq:Mestimation}). 
\begin{lemma}\label{lemma:AMP-Minimizers}
Let  $(\htheta_*,\res_*,b_*)$ be a fixed point of the AMP iteration
(\ref{eq:AMP1}), (\ref{eq:AMPb}), (\ref{eq:AMP2}) having $b_*>0$. Then  $\htheta_*$ is a minimizer of
the problem (\ref{eq:Mestimation}). Viceversa, any minimizer
$\htheta_*$  of the problem (\ref{eq:Mestimation}) corresponds to one
(or more) AMP fixed points of the form $(\htheta_*,\res_*,b_*)$.
\end{lemma}

\begin{proof}
By differentiating Eq.~(\ref{eq:Mestimation}), and omitting the
arguments $Y,\bX$ for simplicity from $\cL(\theta;Y,\bX)$, we get
\begin{align}
\nabla_{\theta}\cL(\theta) = -\sum_{i=1}^n
\rho'\big(Y_i-\<X_i,\theta\>\big)\, X_i = - \bX^{\sT} \rho'(Y-\bX\theta)\, ,\label{eq:Stationarity}
\end{align}
where as usual $\rho'$ is applied component-wise to vector arguments.
The minimizers of $\cL(\theta)$ are all the vectors $\theta$ for which
the right hand side vanishes.

Consider then a fixed point $(\htheta_*,\res_*,b_*)$, of the AMP
iteration (\ref{eq:AMP1}), (\ref{eq:AMP2}). This satisfies the
equations
\begin{align}
\res_* & = Y -\bX\htheta_*+\dual(\res_*;b_*)\, ,\\
0 & = \delta\bX^{\sT}\dual(\res_*;b_*)\, .\label{eq:FixedPoint2S}
\end{align}
The first equation can be written
as
\begin{align}
Y-\bX\theta_* =  \res_* - \dual(\res_*;b_*)\, , \label{eq:ProxFix}
\end{align}
Using Proposition \ref{propo:Phi} below, (\ref{eq:ProxFix})
implies that  $\dual(\res_*;b_*) =  b_* \rho'(Y-\bX\htheta_*)$. 
Hence  the second equation reads
\begin{align}
0 & = \delta b_*\bX^{\sT}\rho'(Y-\bX\htheta_*)\, ,
\end{align}
which coincides with the stationarity condition
(\ref{eq:Stationarity}) for $b_*>0$. This concludes the proof.
\end{proof}

\subsection{Example}

To make the AMP algorithm concrete, we consider an example
with $n=1000$, $p=200$, so $\delta=5$.  For design matrix
we let $X_{i,j} \sim \normal(0,\frac{1}{n})$, and we draw $\theta_0$
a random vector of norm $\| \theta_0 \|_2 = 6 \sqrt{p} $.
For the distribution $F = F_W$ of errors,
we use Huber's contaminated normal
distribution $\Cnormal(0.05,10)$, 
so that $F = 0.95 \Phi + 0.05  H_{10}$,
where $H_x$ denotes a unit atom at $x$.
For the loss function, we use the Huber's $\Huber(z;\lambda)$ with $\lambda= 3$.
Starting the AMP algorithm 
with $\htheta^0 = 0$, we run 20 iterations.

Separately, we solved the M-estimation problem 
using CVX, obtaining $\htheta$.

Figure \ref{fig:AMPHist} (left panel) shows the progress of the AMP
algorithm across iterations, presenting
\[
\RMSE(\htheta^t;\theta_0) \equiv \frac{1}{\sqrt{p}}\,\|\htheta^t-\theta_0\|_2\, ,
\] 
while Figure \ref{fig:AMPHist} (right panel) shows the progress of AMP
in approaching the M-estimate $\htheta$, as measured by
\[
  \RMSE(\htheta^t; \htheta) \equiv \frac{1}{\sqrt{p}}\,\|\htheta^t-\htheta\|_2.
\]
As is evident, the iterations converge rapidly, and they
converge to the M-estimator, both in the sense of convergence of risks  \-- measured here by 
$\RMSE(\htheta^t;\theta_0) \goto \RMSE(\htheta; \theta_0) \approx 1.6182$ \--  and,
more directly, in convergence of the estimates themselves: $\RMSE(\htheta^t; \htheta) \goto 0$.

\begin{figure}
\begin{center}
\includegraphics[height=4in]{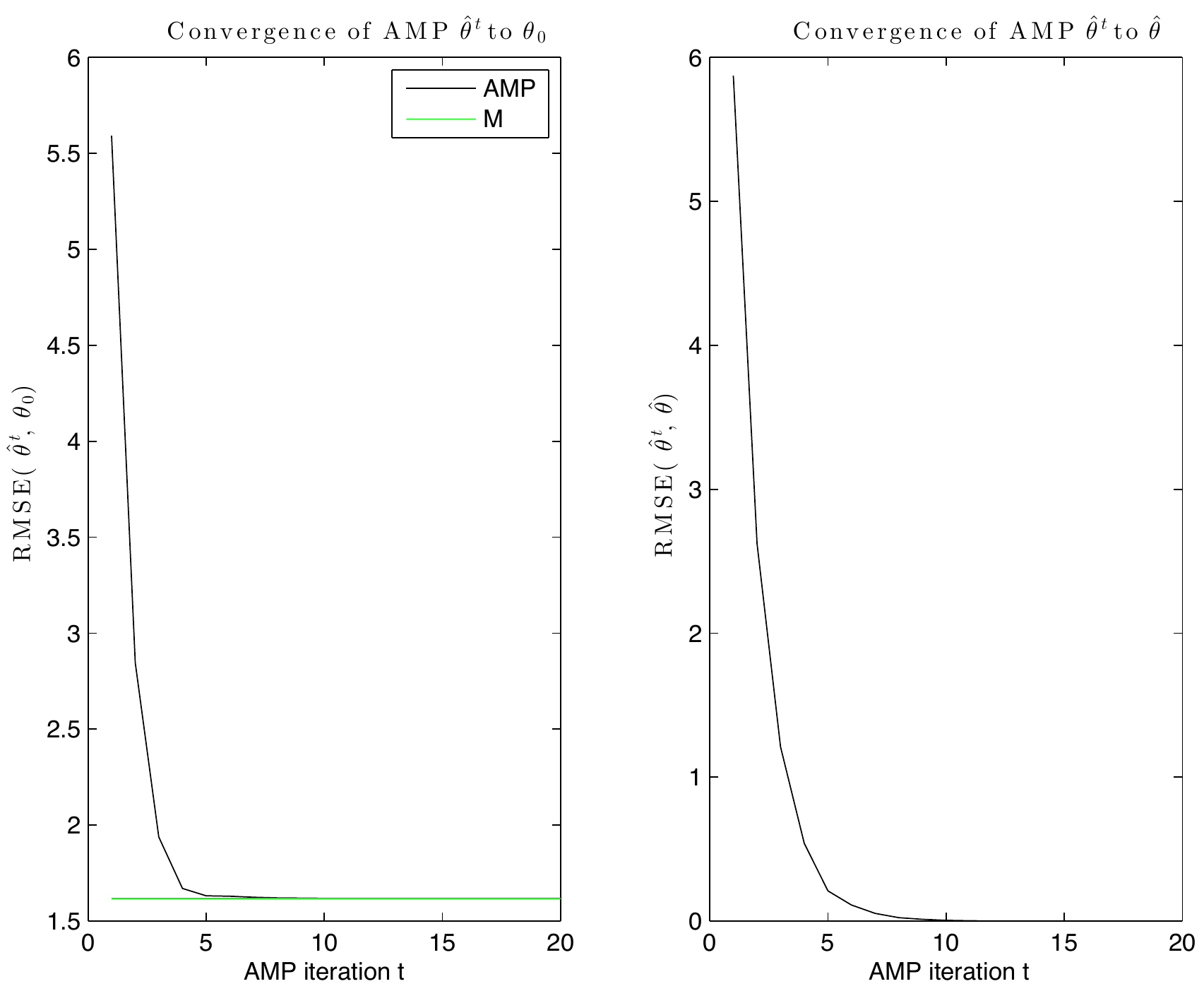}
\caption{Left Panel: RMSE of AMP versus iteration (black curve), and its convergence to RMSE of M-estimation (constant green curve).
Right Panel: Discrepancy of AMP from M-estimate, versus iteration. }
\label{fig:AMPHist}
\end{center}
\end{figure}

Figure \ref{fig:bHist} (left panel) shows the process by which the effective score parameter $\hat{b}_t$
is obtained at iteration $t=3$, while the right panel shows how  $\hat{b}_t$ 
behaves across iterations. In fact it converges quickly towards a limit $b_\infty  \approx 0.2710$.
\begin{figure}
\begin{center}
\includegraphics[height=4in]{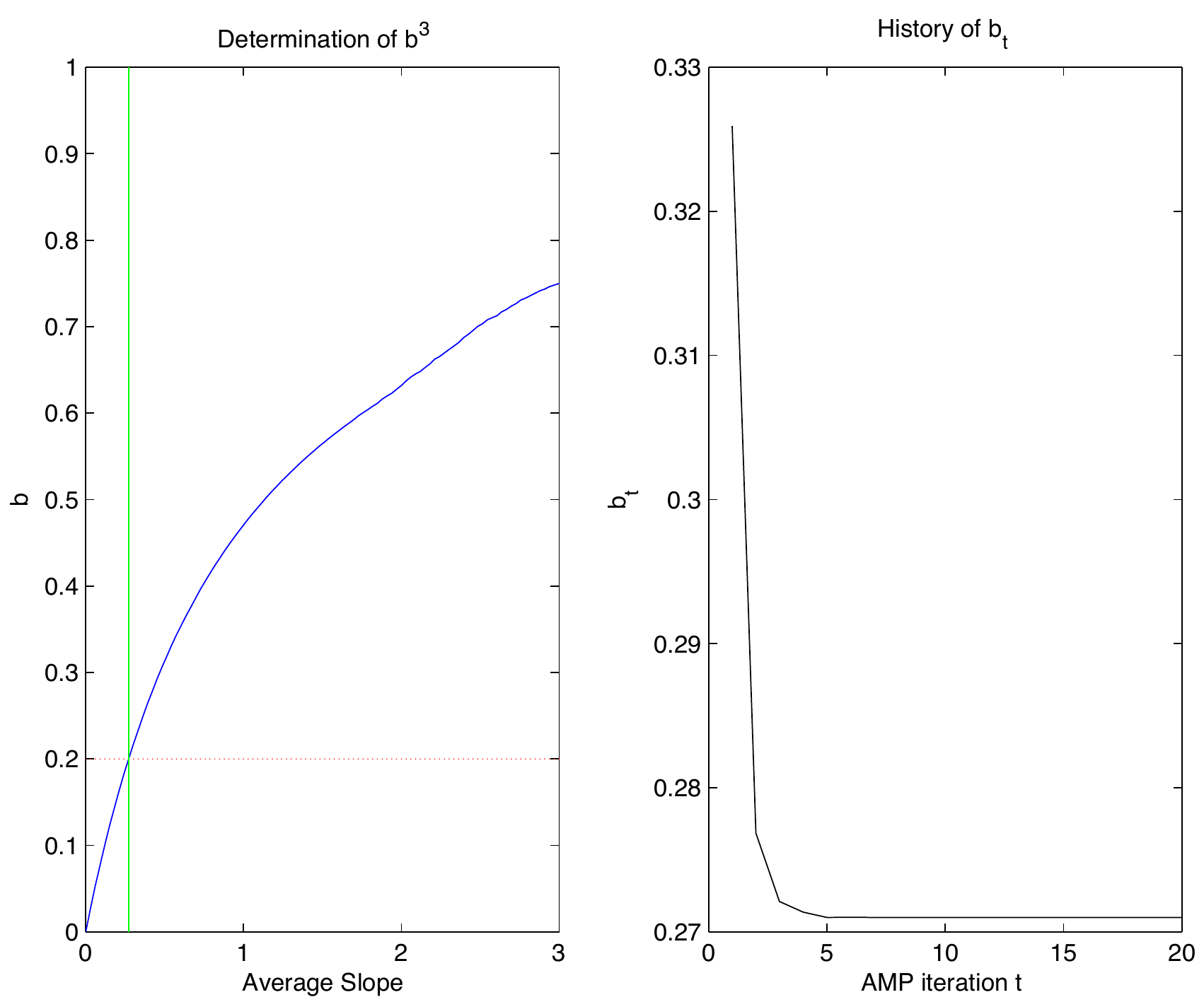}
\caption{Left Panel: Determining the regularization parameter at iteration 3.
Blue curve: Average slope (vertical) versus regularization parameter $b$ (horizontal).
The blue curve intersects desired level $0.2 =1/\delta$ near 0.3. 
Right Panel: regularization parameter $b_t$ versus iteration; it converges rapidly to roughly $0.2710$. }
\label{fig:bHist}
\end{center}
\end{figure}
\subsection{Contrast to iterative M-estimation}

Earlier we pointed to  resemblances between AMP (\ref{eq:AMP2}) and the traditional 
method of scoring for obtaining M-estimators (\ref{eq:TradMethScore}).
In reality the two approaches are very different:
\begin{itemize}
\item The precise form of various terms in (\ref{eq:AMP1}), (\ref{eq:AMPb})  (\ref{eq:AMP2}) is
dictated by the statistical assumptions that we are making on the
design $\bX$. In particular the memory terms are crucial for the state
evolution analysis to hold. Several papers document this point
\cite{MontanariChapter,SchniterTurbo,Schniter-NonUniform-2010,RanganGAMP,krzakala2013phase}.
\item Under classical asymptotics, where $p$ is fixed and $n \goto \infty$, 
  it is sufficient to run a single
  step of such an algorithm \cite{bickel1975}, in the high-dimensional
  setting it is necessary to iterate numerous times. The resulting
  analysis is considerably more complex because of correlations
  arising as the algorithm evolves. 
\end{itemize}

%
\section{State evolution description of AMP}
\label{sec:StateEvolution}

State Evolution is a method for computing 
the operating characteristics of the AMP iterates $\htheta^t$ and $\res^t$
for arbitrary fixed $t$,  under the high-dimensional asymptotic limit $n,p\to\infty$,
$n/p \goto \delta$. 

In this section we initially describe a purely formal procedure
which \emph{assumes} that  the AMP adjusted residuals
$\res^t = Y - \bX \htheta^t + \Psi(\res^t;b_t)$ really behave as $W + \tau_t Z$,
with $W$ the error distribution and $Z$ an independent standard normal,
for $t=0,1,2,\dots$.
The variable $\tau_t^2$ thus quantifies the extra Gaussian noise
supposedly present in the adjusted residuals of AMP; 
 we  show how this ansatz allows one to calculate $\tau_t^2$
for each $t=0,1,2,3, \dots$, and to calculate the limit of
$\tau_t$ as $t \goto \infty$. Later in the section we present a
rigorous result
validating the method under the following random Gaussian design assumption.
\begin{definition} \label{def-GaussDesign}
We say that a sequence of random design matrices $\{\bX(n)\}_n$,
 with $n\to\infty$ is a \emph{Gaussian design} if each
 $\bX=\bX(n)$ has dimensions $n\times p$, and entries $(X_{ij})_{i\in
    [n],j\in [p]}$ that are i.i.d. $\normal(0,1/n)$. Further, $p = p(n)$
  is such
that $\lim_{n\to\infty} n/p(n) = \delta\in (0,\infty)$.
\end{definition}

\subsection{Initialization of the extra variance}

\newcommand{\bu}{{\bf u}}
Under the Gaussian design assumption,
suppose that $\bu$ is a vector in $\reals^p$ with norm $\| u \|_2$.
Then $\{\E \| \bX u \|_2^2\} = \|u\|_2^2$. Moreover, $X\bu$ is a Gaussian
random vector with entries iid $\normal(0,\|u\|_2^2/n)$.

It will be convenient to introduce for any estimator $\tilde{\theta}$ the notation
\beq \label{eq:defMSE}
      \MSE(\tilde{\theta},\theta_0) =  \frac{1}{p}\, m \| \tilde{\theta} - \theta_0 \|_2^2 .
\eeq

So initialize
AMP with a deterministic estimate $\htheta^0$, and take $\res^{-1}=0$. 
Then the initial residual is
$\res^1 = Y - \bX \htheta^0 = W + \bX ( \theta_0 - \htheta^0)$.
The terms $W$ and  $\bX ( \theta_0 - \htheta^0)$ are independent,
and  $\bX ( \theta_0 - \htheta^0)$ is Gaussian with variance
$\tau_0^2 =   \|\htheta^0 - \theta_0\|_2^2/n = \MSE(\htheta^0, \theta_0)/\delta$.
Consider some fixed coordinate $\res^1(i)$ of $\res^1$. Then
\[
   \Var(\res^1_i) = \Var(W) + \Var(\bX (\theta_0 - \htheta^0)) = \Var(W) + \MSE(\theta^0, \theta_0)/\delta .
\] 
Hence, when AMP is started this way, we see that
the adjusted residuals initially contain an extra Gaussian noise
of variance $\tau_0^2 = \MSE(\htheta^0,\theta_0)/\delta$. 

\subsection{Evolution of the extra Gaussian variance to its ultimate limit}

Assuming 
the adjusted residuals continue, at later iterations, to behave as $W + \tau_t\, Z$
with $Z$ an independent standard normal,
 we now calculate $\tau_t^2$
for each $t=1,2,3, \dots$, and eventually identify the limit of
$\tau_t$ as $t \goto \infty$.

For a given $\tau > 0$, $\delta = n/p$ and noise distribution $F_W$,
define the \emph{variance map}
\[
   \cV(\tau^2 , b; \delta, F_W)=   \delta \, \E\Big\{\dual(W+\tau\, Z;b)^2\Big\}\, ,
\]
where $W\sim F_W$, and, independently, $Z\sim\normal(0,1)$.
In this display, the reader can see that extra Gaussian noise of
variance $\tau^2$ is being added to the underlying noise $W$,
and $\cV$ measures the $\delta$-scaled variance of the
resulting output.
Evidently for $b > 0$, $0 \leq \cV(\tau^2, b) \cdot \delta  \leq (\Var(W)+\tau^2) \cdot \delta$.

Under our assumptions for $\Psi$,
for each given specification $(\tau;\delta,F_W)$ of 
the ingredients besides $b$ that go into $\cV$, there is (as clarified by Lemma \ref{lemma:BtDef})
a well-defined value  $b = b(\tau;\delta,F_W)$ giving  the smallest
solution $ b \geq 0 $ to
\begin{align}
\frac{1}{\delta} = \E\Big\{\Psi'(W+\tau \cdot \,Z;b)\Big\}\, .\label{eq:BtDef}
\end{align}

\newcommand{\tcV}{\tilde{\cal V}}
\begin{definition}
{\em State Evolution} is an iterative process for computing the scalars
$\{\tau^2_t\}_{t\ge  0}$,  starting from an initial condition 
$\tau_0^2\in \reals_{\ge 0}$ following
\begin{align}
\tau_{t+1}^2 =  \cV (\tau_t^2 , b(\tau_t) ) =  \cV(\tau_t , b(\tau_t; \delta,F_W) ; \delta, F_W) .\label{eq:StateEvolution}
\end{align}
\end{definition}

Defining  $\tcV(\tau^2)=  \cV (\tau^2 , b(\tau) )$, we see that the evolution of $\tau_t$
follows the iterations of the map $\tcV$. In particular, we make these
observations:
\bitem
 \item $\tcV(0)  > 0 $,
 \item $\tcV(\tau^2)$ is a continuous, nondecreasing function of $\tau$.
 \item $\tcV(\tau^2) < \tau^2$ as $\tau \goto \infty$.
\eitem

Figure \ref{fig:SEHist}, left panel, considers the case where $W$
again follows the Huber's contaminated
normal distribution $\Cnormal(0.05,10)$ and $\psi$ is the standard Huber estimator
with parameter $\lambda=3$. The ratio $n/p = \delta = 2$, and the parameter vector has $ \|\theta_0 \|_2^2/p = 6^2$.
It displays the function $\tilde{V}(\tau^2)$ as a function of $\tau$.

\begin{figure}
\begin{center}
\includegraphics[height=2.5in]{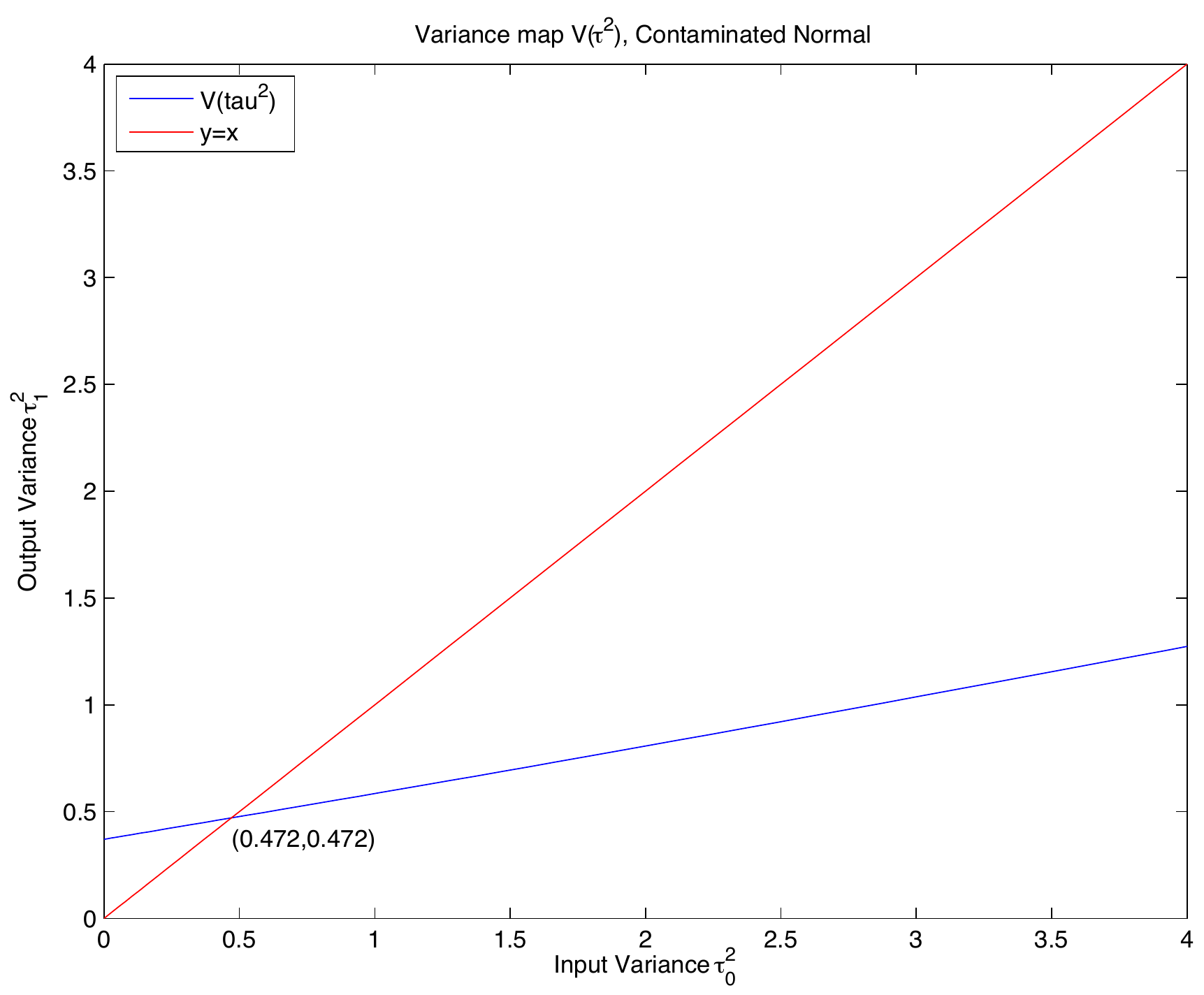} \qquad 
\includegraphics[height=2.5in]{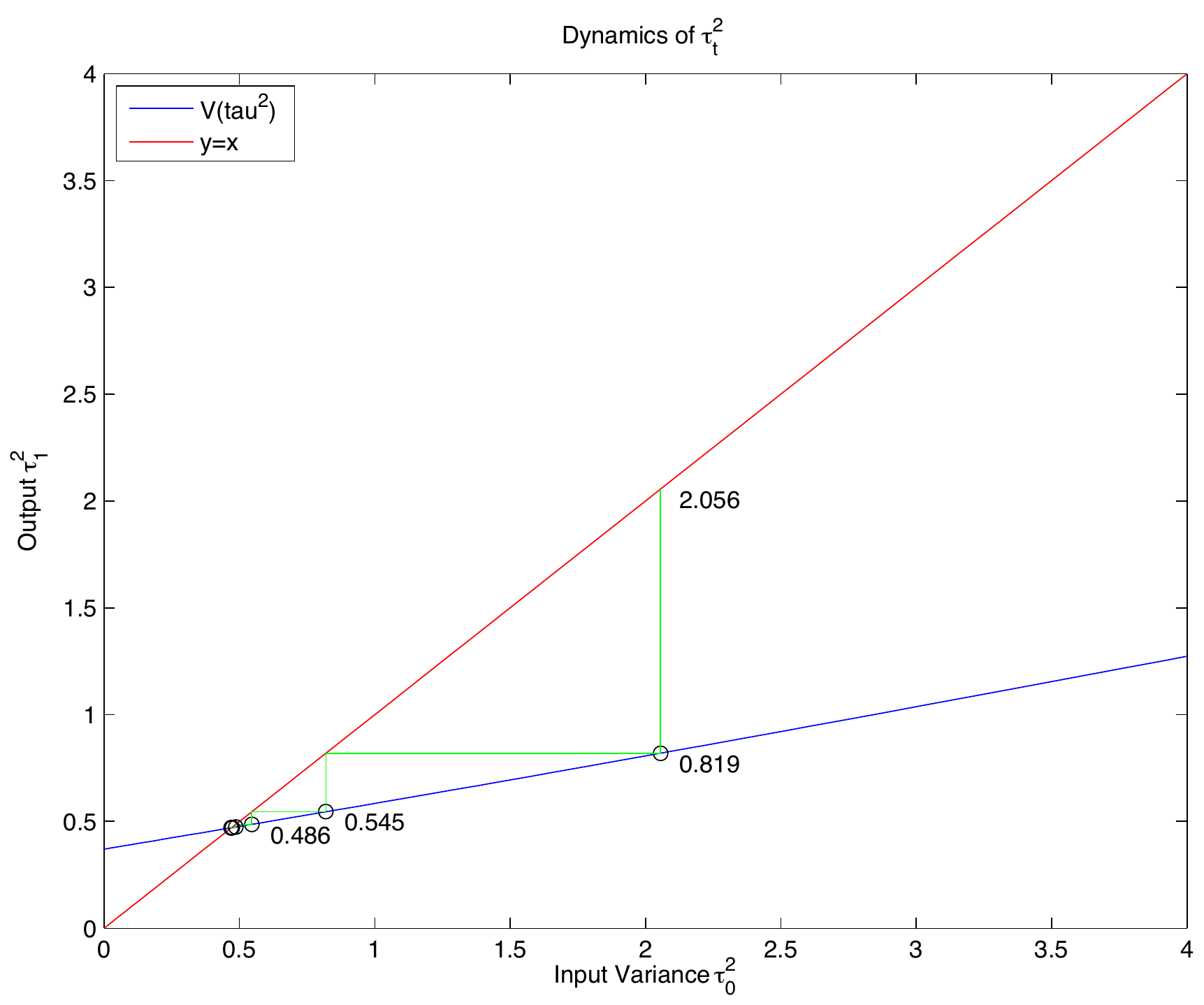} 
\caption{The State Evolution Variance Mapping. Left Panel: Blue Curve: $\tcV$ versus $\tau^2$, Red Curve: diagonal; unique fixed point
at about $0.472$.
Right Panel:  the iteration history of state evolution, starting from $\tau_0^2 = 2.0556$ }
\label{fig:SEHist}
\end{center}
\end{figure}

Evidently, there is a stable fixed point $\tau_* = \tau_*(\delta,F_W)$, i.e. a point obeying
$\tcV(\tau_*^2) = \tau_*^2$,  such that  $\tau^2 \mapsto \tcV(\tau^2)$ has 
a derivative less than $1$ at $\tau_*^2$.
We conclude that $\tau_t$ evolves under state evolution to
a nonzero limit. 
 Figure \ref{fig:SEHist}, right panel, shows how $\tau_t^2$ evolves to the fixed point near $0.472$ 
starting from $\tau_0^2 = 2.056$.

%

\subsection{Predicting operating characteristics from State Evolution}

State Evolution offers a  formal\footnote{By \emph{formal},  we mean a rule-based procedure
which we can follow to get a prediction, without any guarantees that the prediction is correct.} procedure for 
predicting operating characteristics
of the AMP iteration at any fixed iteration $t$
or in the limit $t \goto \infty$. Nater in this section, we will
provide rigorous validation of these predictions.

Call  the tuple $S = (\tau ; b, \delta, F)$ a {\em state}; in running the AMP
algorithm we  assume that the algorithm is initialized with $\htheta^0$
so that $\tau_0^2 = \MSE(\htheta^0,\theta_0)/\delta$, so that AMP starts in state $S = (\tau_0; b^0,\delta,F)$,
and visits $S_1 = (\tau_1;b^1,\delta,F)$, $S_ 2 = (\tau_2;b^2,\delta,F)$, \dots; eventually AMP 
visits states arbitrarily close to the equilibrium state $S_*= (\tau_*; b^*, \delta, F)$.

SE predictions of operating characteristics are provided by two rules
assigning predictions to certain classes of observables, based on the state
that AMP is in.

\begin{definition} \label{def:SEFormalism}
The {\em state evolution formalism} assigns predictions $\cE$ to 
two types of observables
under specific states. 
\begin{description}
\item [Observables Involving $\htheta -\theta_0$. ]
Given a  univariate test function $\test : \reals \mapsto \reals$,
assign the predicted value for 
$p^{-1}\sum_{i\in p}\test( \htheta_i -  \theta_{0,i})$ under  state $S$ by the rule
\[
 \cE(\test(\htheta - \vartheta) | S) \equiv   \E\Big\{\test(\sqrt{\delta}\,\tau\, Z)\Big\}\,  ,
\]
where expectation on the right hand side is with respect to $Z \sim \normal(0,1)$.
\item [Observables involving Residual, Error. ] Let $R$ denote some coordinate of
the  adjusted residual for AMP in state $S$ and $W$ the same coordinate of the
underlying error. 
Given a  bivariate test function $\test_2 : \reals^2 \mapsto \reals$,
assign the prediction of $n^{-1}\sum_{i=1}^{n}\test_2(R_i,W_i)$ in  state $S$ by
\[
 \cE(\test_2(R,W) | S) \equiv   \E  \test_2(W+\tau\, Z,W) 
\]
where $Z \sim \normal(0,1)$ and $W \sim F_W$ is independent of $Z$.
\end{description}
\end{definition}

The two most important predictions of  operating characteristics
are undoubtedly:
\bitem
 \item \emph{$\MSE$ at iteration $t$}. 
We let $S_t = (\tau_t, b(\tau_t), \delta, F_W)$ denote the
 state of AMP at iteration $t$, and predict
 \[
     \MSE(\htheta^t, \theta_0) \approx  \cE((\hat{\vartheta} - \vartheta)^2 | S_t) = \E\Big\{(\sqrt{\delta}\,\tau_t\, Z)^2 \Big\} = \delta \tau_t^2.
 \]
 \item \emph{$\MSE$ at convergence}. 
With $\tau_* > 0 $ the limit of  $\tau_t$,
let $S_* = (\tau_*, b(\tau_*), \delta, F_W)$ denote the
 state of AMP at  convergence.  and predict
 \[
     \MSE(\htheta_*, \theta_0) \approx  \cE((\hat{\vartheta} - \vartheta)^2 | S_*) = \E\Big\{(\sqrt{\delta}\,\tau_*\, Z)^2  \Big\}= \delta \tau_*^2.
 \]
 \eitem 
Other predictions might also be of interest. Thus, concerning the mean absolute error
 $\MAE(\htheta^t,\theta_0) = \| \htheta^t - \theta_0 \|_1/p$, state
 evolution predicts
 $\MAE \approx \sqrt{2 \delta\tau_t^2/\pi}$.
 Concerning functions of $(R,W)$,
 consider the ordinary residuals $Y - X \htheta^*$ at AMP convergence.
 These residuals will of course in general not have
 the distribution of the errors $W$.
 Setting $\eta( z ; b) = z - \dual( z; b)$, we have $Y - X \htheta^* = \eta(R; b_*)$.
  State evolution predicts that the ordinary residuals will
 have the same distribution as $\eta( W + \tau_* Z ; b_*)$.
 
 \subsection{Example of State Evolution predictions}

Continuing with our running example, we again
consider the case of contaminated normal
data  $W \sim \Cnormal(0.05,10)$ and Huber $\rho$ with $\lambda=3$.
If we start AMP with the all-zero estimate $\htheta^0 = 0$,
then since $ \| \theta_0 \|_2 = 6 \sqrt{p}$ we start SE with 
$\tau_0 = 2.056$. Figure \ref{fig:SEPred} presents predictions
by state evolution for the MSE (left panel) and
for the mean absolute error MAE.
\begin{figure}
\begin{center}
\includegraphics[height=3.25in]{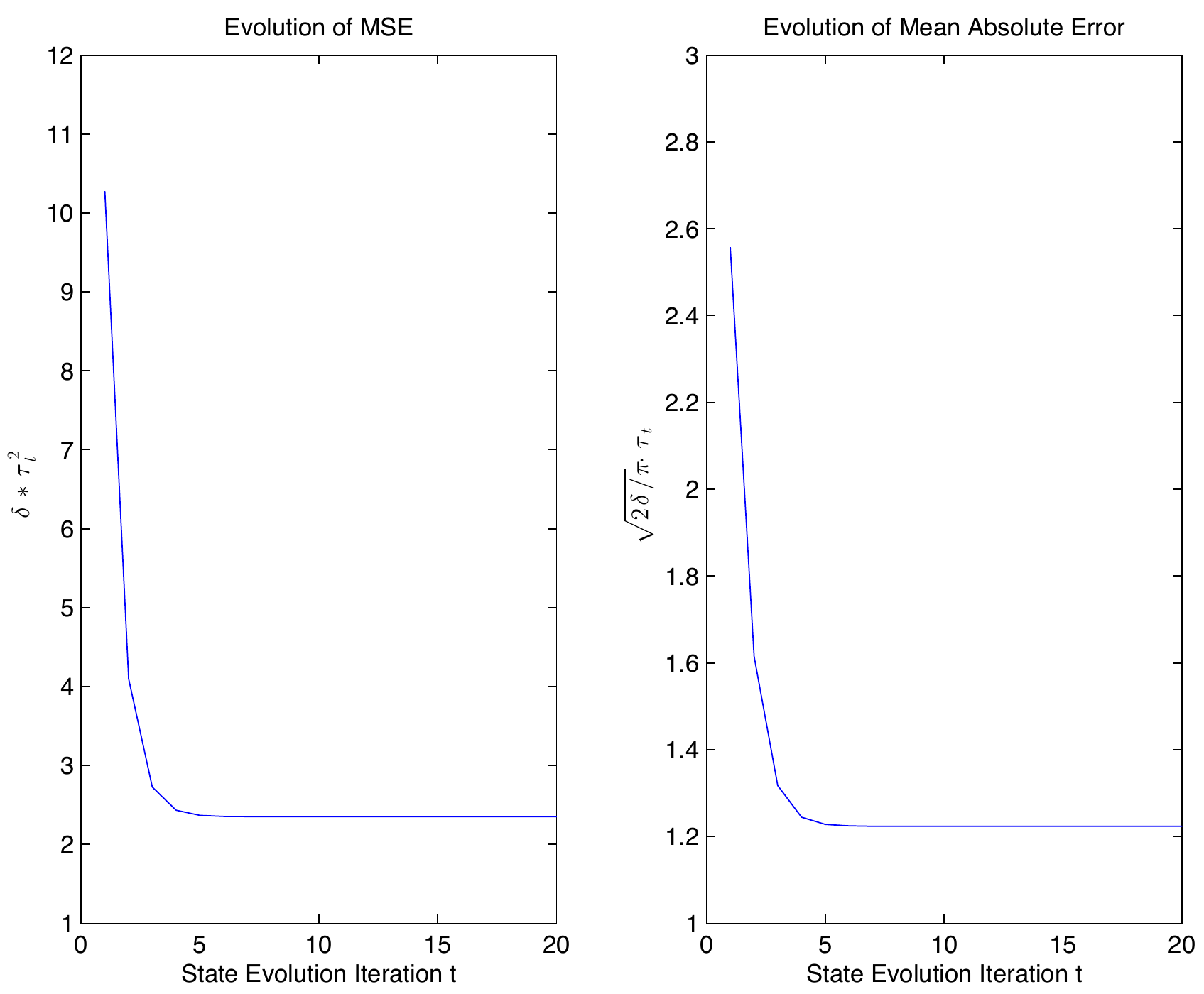}
\caption{State Evolution predictions for $\Cnormal(0.05,10)$, with Huber $\psi$, $\lambda=3$.
Predicted evolutions of two observables of $\htheta^t-\theta_0$:  Left: MSE, Mean Squared Error.
Right: MAE, Mean Absolute Error.
}
\label{fig:SEPred}
\end{center}
\end{figure}

Again in our running example, these predictions can be tested empirically.
For illustration, we conducted a very small experiment,
generating 10 independent realizations of the running model at $n=1000$ and $p=200$, and comparing the actual evolutions
of observables during AMP iterations with the predicted evolutions.
Figure \ref{fig:ObsMeanPred}
shows that the predictions from SE are very close to the averages across realizations. 

\begin{figure}
\begin{center}
\includegraphics[height=3.75in]{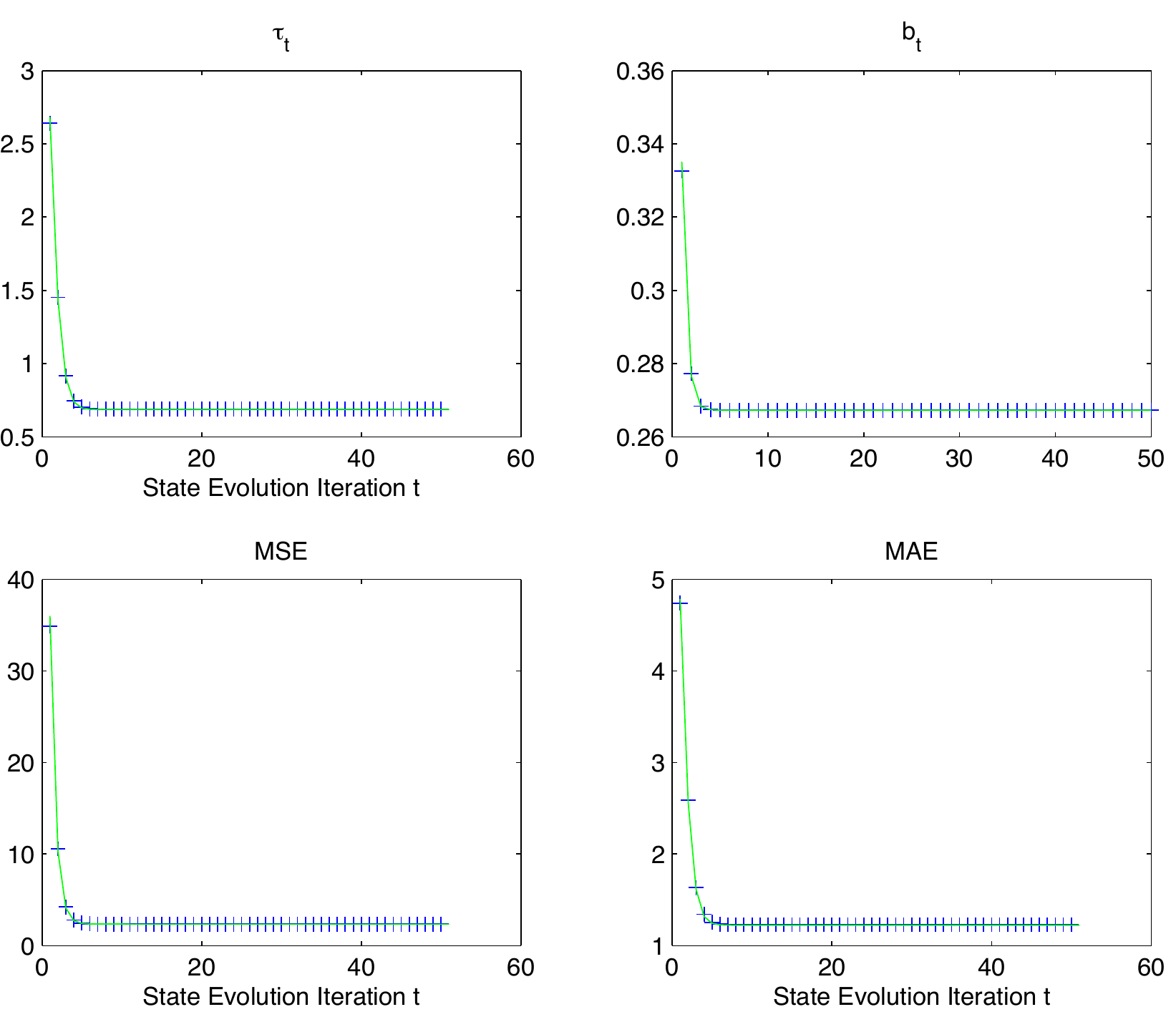}
\caption{Experimental means from 10 simulations
compared with State Evolution predictions under $\Cnormal(0.05,10)$, with Huber $\psi$, $\lambda=3$.
Upper Left: $ \hat{\tau}_t = \| \htheta^t - \theta_0 \|_2/\sqrt{n}$.
Upper Right: $\hat{b}_t$.
Lower Left: MSE, Mean Squared Error.
Lower Right: MAE, Mean Absolute Error.
Blue `+' symbols: Empirical means of AMP observables.
Green Curve: Theoretical predictions by SE.
}
\label{fig:ObsMeanPred}
\end{center}
\end{figure}

\subsection{A lower bound on State Evolution}

State Evolution cannot evolve so that $\tau_t^2 \goto 0$; under minimal regularity, it always exceeds
a specific nonzero noise level. 

\begin{lemma} \label{lem:SELB}
Suppose that $F_W$ has a well-defined Fisher information $I(F_W)$. Then for any 
$t > 0$
\[
     \tau_t^2 \geq \frac{1}{\delta I(F_W)}.
\]
\end{lemma}

\begin{proof}
Let $G = F_W \star N(0,\tau^2)$.  $\E_G \Psi' = \frac{1}{\delta}$ and, if $\xi_G$ denotes the
score function for location of $G$, then $| \E_G \Psi' | = |\E_G \Psi \cdot \xi_G|$ . Meanwhile,
by Cauchy-Schwartz, $ |\E_G \Psi \cdot \xi_G| \leq \sqrt{ \E_G \Psi^2 } \sqrt{\E_G \xi_G^2}$.
We conclude that
\[
    \tcV(\tau)  = \delta \E_G \Psi^2 \geq \delta \frac{|\E_G \Psi \cdot \xi_G|^2}{\E_G \xi_G^2} =  \delta \frac{|\E_G \Psi' |^2}{I(G)} = \frac{1}{\delta I(G)} .
\]
From convexity and translation-invariance of Fisher Information $I(G) = I(F_W \star \normal(0,\tau^2)) <  I(F_W)$.
Then $\tau_t^2 = \tcV(\tau_{t-1}^2) \geq 1/(\delta I(F_W))$.
\end{proof}

We can sharpen this bound one step further. It will be convenient to
write $I(X)$ for the Fisher information of distribution $F_X$.

\begin{lemma}
\[
     I(W+ \tau Z) \leq \frac{I(W)}{1 + \tau^2 I(W)}.
\]
\end{lemma}

\begin{proof}
Barron and Madiman \cite{BarronMadiman2007} give the inequality $I(W + \tau Z) \leq  x^2 I(W) + (1-x)^2 I(\tau Z)$,
valid for any $x \in (0,1)$. By calculus, we know that for $a,b > 0$,
\[
     \min_{x \in (0,1)}  x^2 a + (1-x)^2 b = \frac{ab}{a+b}.
\]
Setting $a=I(W)$ and $b=I(\tau Z) = \tau^{-2}$, and dividing both numerator and denominator by $b$, we are done.
\end{proof}

Revisit the argument of Lemma \ref{lem:SELB};
the inequality $\tau_{t}^2 \geq 1/(\delta I(F_W))$ shows that if $t > 0$, then 
 $ \tau_{t}^2 I(W) \geq 1/\delta$. Using this in the previous Lemma,
\[
 I(W+ \tau_{t} Z) \leq \frac{I(W)}{1 + \tau_{t}^2 I(W)} \leq \frac{I(W)}{1 + \frac{1}{\delta}}.
\]
This yields a `one-step' improvement:
\begin{coro} \label{coro:SE1SLB}
Suppose that $F_W$ has a well-defined Fisher Information $I(F_W)$. Then for any 
$t > 1$
\[
     \tau_t^2 \geq \frac{1 + \frac{1}{\delta}}{\delta I(F_W)}.
\]
\end{coro}

We can iterate this argument across many steps, obtaining that, for every $t > k$,
\[
    \tau_t^2 \geq \frac{1 + \frac{1}{\delta} + \frac{1}{\delta^2} + \dots +  \frac{1}{\delta^k}}{\delta I(F_W)}.
\]
We obtain immediately:
\begin{coro} \label{coro:SEInftyLB}
Suppose that $F_W$ has a well-defined Fisher information $I(F_W)$. Then for every accumulation point  $\tau_*$ of
State Evolution
\[
     \tau_*^2 \geq \frac{1}{\delta -1} \cdot \frac{1}{ I(F_W)}.
\]
\end{coro}

\subsection{Correctness of State Evolution predictions}

The predictions of state evolution can be validated in the large-system limit $n,p \goto \infty$,
under the random Gaussian design assumption of Definition \ref{def-GaussDesign}.
We impose regularity conditions on 
the observables whose behavior we attempt to predict:
\begin{definition}
A function $\xi:\reals^k\to\reals$ is \emph{pseudo-Lipschitz} if there
exists $L<\infty$ such that, for all $x,y\in\reals^k$,
$|\xi(x)-\xi(y)|\le L(1+\|x\|_2+\|y\|_2)\,\|x-y\|_2$.
\end{definition}
In particular, $\xi(x) = x^2$ is pseudo-Lipschitz.

Recall also the definition of MSE in equation
(\ref{eq:defMSE}). For a sequence of  estimators $\tilde{\theta}$,
define the per-coordinate asymptotic mean squared error (AMSE)
as the following large-system limit:
\begin{align}
\AMSE(\tilde{\theta};\theta_0) =_{{\rm a.s.}}  \lim_{n,p_n \goto \infty} \MSE(\tilde{\theta};\theta_0),
\end{align}
when the indicated limit exists.

The following result validates the predictions of State Evolution
for pseudo-Lipschitz observables. Our proof is deferred to 
Appendix \ref{app:StateEvolution}. 
\begin{thm}\label{thm:StateEvolution}
Assume that the loss function $\rho$ is convex and smooth, that the
sequence of matrices $\{\bX(n)\}_{n}$ is a standard Gaussian design,
and that $\theta_{0}$, $\htheta^0$ are deterministic sequences such that 
$\AMSE(\theta_{0}$, $\htheta^0)  = \delta\tau_0^2$.
 Further assume that $F_W$ has
finite second moment and let $\{\tau_t^2\}_{t\ge 0}$ be the state
evolution sequence with initial condition $\tau_0^2$. Let
$\{\htheta^t,\res^t\}_{t\ge 0}$ be the AMP trajectory with parameters
$b_t$ as per Eq.~(\ref{eq:BtDef}).

Let $\test:\reals\to \reals$, $\test_2:\reals\times\reals\to \reals$
be   pseudo-Lipschitz functions.
Then, for any $t>0$,  we have, for
$Z\sim\normal(0,1)$
independent of $W\sim F_W$
\begin{align}
\lim_{n\to\infty}\frac{1}{p}\sum_{i=1}^p\test(\htheta_i^t-\theta_{0,i})
=_{a.s.} \,  &\E\Big\{\test(\sqrt{\delta}\,\tau_t\, Z)\Big\}\, ,\\
\lim_{n\to\infty}\frac{1}{n}\sum_{i=1}^n\test_2(\res^t_i,W_i)
=_{a.s.} \, & \E\Big\{\test_2(W+\tau_t\, Z,W)\Big\}\,.
\end{align} 
\end{thm}
%
%
In particular, we may take $\xi(x) = x^2$ and obtain  for the AMP iteration
\[
      \AMSE(\htheta^t,\theta_0) = \delta \tau_t^2,
\]
in full agreement with the predictions of state evolution in Definition \ref{def:SEFormalism}.

\section{Convergence and characterization of M-estimators}

The key step for characterizing the distribution of the M-estimator
$\htheta$, cf. Eq.~(\ref{eq:Mestimation}), is to prove that the AMP
iterates $\htheta^t$
converge to $\htheta$. We will prove that this is indeed the case, at
least in the limit $n,p\to\infty$, and for suitable initial
conditions\footnote{We expect convergence for arbitrary initial conditions
(as long as they are independent of $(W,\bX)$), but proving this claim
is not needed for our main goal,  and we leave it for future study.
Proving this claim would require showing convergence of the state evolution
recursion (\ref{eq:StateEvolution}).}.

Throughout this section, we shall assume that $\rho$ is \emph{strongly
convex}, i.e. that $\inf_{x\in\reals}\rho''(x)>0$. 
This corresponds to assuming $\inf_{x\in\reals}\psi'(x)>0$,
which is rather natural from the point of view of robust statistics
since it ensures uniqueness of the M estimator\footnote{The Huber estimator is not covered by the result of this section;
although we expect our
approach to apply in such generality. We focus here on the strongly
convex case to avoid un-necessary complications.}.

The key step is to establish the following high-dimensional
convergence result.
\begin{thm}\label{thm:Convergence}
 ({\bf Convergence of AMP to the M-Estimator.})
Assume the same setting as in Theorem \ref{thm:StateEvolution}, and
further assume that $\rho$ is strongly convex and  that $\delta>1$.

Let $(\tau_*,b_*)$ be a solution of the two equations
\begin{align}
\tau^2 &= \delta \; \E\Big\{\dual(W+\tau\, Z;b)^2\Big\}\, ,\label{eq:FixedPoint1}\\
\frac{1}{\delta} &= \E\Big\{\dual'(W+\tau\,Z;b)\Big\}\, . \label{eq:FixedPoint2}
\end{align}
and assume that $\AMSE(\htheta^0,\theta_0)= \delta\tau_*^2$. Then
\begin{align}
\lim_{t\to\infty}
\AMSE(\htheta^t,\htheta) = 0\, .
\end{align}
\end{thm}

From this and Theorem \ref{thm:StateEvolution}, the desired characterization of $\htheta$ immediately follows.

To tie back to the introduction, we
prove formula (\ref{HDAVarCov}):
\begin{coro}\label{coro:AysVariance} ({\bf Asymptotic Variance Formula under High-Dimensional Asymptotics.})
Assume the  setting  of Theorem \ref{thm:StateEvolution}, and
further assume that $\rho$ is strongly convex and $\delta>1$. 
The asymptotic variance of $\htheta$ obeys
\begin{align}
\lim_{n,p \to \infty} {\rm Ave}_{i\in[p]}\Var(\htheta_i) =_{{\rm a.s}} V(\tilde{\Psi},\tilde{F}), 
\end{align}
where ${\rm Ave}_{i\in[p]}$ denotes the average across indices $i$, $V(\psi,F)$ denotes the usual Huber asymptotic variance formula for M-estimates --
$V(\psi,F) = (\int \psi^2 \de F)/(\int \psi' \de F)^2$ --
 and the effective score $\tilde{\Psi}$ is
\[
      \tilde{\Psi}(\,\cdot\,) = \Psi(\,\cdot\, ; b_*) ,
\]
while the effective noise distribution $\tilde{F}$ is
\[
     \tilde{F} = F_W \star \normal(0,{\tau}_*^2).
\]
Here $(\tau_*,b_*)$ are the unique
solutions of the equations (\ref{eq:FixedPoint1})-(\ref{eq:FixedPoint2}).
\end{coro}

\begin{proof}
By symmetry, ${\rm Ave}_{i\in [p]} \Var(\htheta_i) = \E \MSE(\htheta,\theta_0)$.
Theorem \ref{thm:Convergence} and State Evolution show that $\AMSE(\htheta,\theta_0) = \delta \tau_*^2$.
By (\ref{eq:FixedPoint1})-(\ref{eq:FixedPoint2})
\[
V(\tilde{\Psi},\tilde{F}) =  \frac{\E \Psi^2(W + \tau_*Z;b_*)}{[\E \Psi'(W + \tau_* Z; b_*)]^2} = \frac{ \tau_*^2/\delta}{\delta^{-2}}   = \delta \tau_*^2.
\]
\end{proof}

Recall that the traditional information bound for M-estimators is $V(\Psi,F) \geq \frac{1}{I(F_W)}$,
and that this is achievable under $p$ fixed, $n \goto \infty$ asymptotics.
Considering the formula for  $\tilde{F}$ we see that 
because $\tau_* > 0$, such an asymptotic variance is
not achievable under high-dimensional asymptotics.
We now make this effect more visible.
Combining Corollary \ref{coro:AysVariance} with Corollary \ref{coro:SEInftyLB}'s
lower bound on the equilibrium noise $\tau_*$ reachable by State Evolution, 
we have the following.

\begin{coro} {\bf Information Bound under High-Dimensional Asymptotics}:
\[
     V(\tilde{\Psi},\tilde{F}) \geq  \frac{1}{1-1/\delta} \cdot  \frac{1}{I(F_W)} .
\]
\end{coro}

In this inequality, the effect of the high-dimensional asymptotics parameter $\delta$ 
is  extremely clear; it shows that the classical
information bound is not achievable when $ \delta = n/p < \infty$,
There is always an inflation in variance  at least by $(1-1/\delta)^{-1} = \frac{n}{n-p}$.
Moreover, the inflation completely blows up as $\delta \to 1$.

\begin{coro}\label{coro:Main}
Assume the  setting of Theorem \ref{thm:StateEvolution}, and
further assume that $\rho$ is strongly convex and $\delta>1$.
Then for  any pseudo-Lipschitz function $\test:\reals\to \reals$, we have, for $Z\sim\normal(0,1)$
\begin{align}
\lim_{n\to\infty}\frac{1}{p}\sum_{i=1}^p\test(\htheta_i^t-\theta_{0,i})
=_{a.s.} \,  \E\Big\{\test(\sqrt{\delta}\; \tau_*\, Z)\Big\}\, .
\end{align} 
In particular, the solution of Eqs.~(\ref{eq:FixedPoint1}),
(\ref{eq:FixedPoint2}) is necessarily unique. 
\end{coro}

Among other applications, this result can be used to bound the
suboptimality of AMP after a fixed number of iterations. 
Combining Theorems
\ref{thm:StateEvolution} and \ref{thm:Convergence} gives:
\begin{coro}\label{coro:Convergence}
Assume the same setting as in Theorem \ref{thm:StateEvolution}, and
further assume that $\rho$ is strongly convex and $\delta>1$. Then
the almost sure limits $\AMSE(\htheta^t;\theta_0)$ and $\AMSE(\htheta;\theta_0)$
exist, and obey
\begin{align}
\AMSE(\htheta^t;\theta_0)-\AMSE(\htheta;\theta_0) =\delta(\tau_t^2-\tau_*^2)\, .
\end{align}
\end{coro}

Theorem \ref{thm:Convergence} extends to cover
general Gaussian matrices  $\bX$ with 
i.i.d. rows.
\begin{definition}
We say that a sequence of random design matrices $\{\bX(n)\}_n$,
 with $n\to\infty$, is \emph{a general Gaussian design} if each
 $\bX=\bX(n)$ has dimensions $n\times p$, and rows $(X_{i})_{i\in
    [n]}$ that are i.i.d. $\normal(0,\Sigma/n)$, where
  $\Sigma=\Sigma(n)\in\reals^{p\times p}$ is a strictly positive
  definite matrix. Further, $p = p(n)$
  is such
that $\lim_{n\to\infty}n/p(n) = \delta\in (0,\infty)$.
\end{definition}

Notice that, if $\bX$ is a general Gaussian design, then
$\bX\Sigma^{-1/2}$ is a standard Gaussian design.
The following then follows from Corollary \ref{coro:MainGen} together
with a simple change of variables argument, cf. \cite[Lemma 1]{karoui2013robust}.
\begin{coro}\label{coro:MainGen} 
Assume the same setting as in Theorem \ref{thm:StateEvolution}, but
with $\{\bX(n)\}_{n\ge 0}$ being a general Gaussian design with
covariance $\Sigma$, and
further assume that $\rho$ is strongly convex and $\delta>1$.
There is a scalar random variable $T_n$ so that 
\begin{align}
\htheta = \theta_0 + \sqrt{\delta}\, T_n \Sigma^{-1/2}\bZ\, ,\label{eq:Representation}
\end{align}
where $\bZ\sim\normal(0,\id_{p\times p})$ and we have the almost-sure limit
$\lim_{n\to\infty}\ T_n =_{a.s.} \tau_*$, where $\tau_*$ solves Eqs.~(\ref{eq:FixedPoint1}),
(\ref{eq:FixedPoint2}).
\end{coro}

This result coincides with \emph{Corollary 1} in
\cite{karoui2013robust} apart from a factor $\sqrt{n}$ in the random
part of Eq.~(\ref{eq:Representation}) that arises because of a
difference in the normalization of $\bX$. 
%
%
\section{Discussion}
\label{sec:Discussion}

Several generalizations of the present proof technique should be
possible, and would be of interest. We list a few in order of
increasing difficulty:
\begin{enumerate}
\item Generalize the i.i.d. Gaussian rows model for $\bX$ by allowing different
  rows to be randomly scaled copies of a common
  $X\sim\normal(0,\Sigma/n)$. This is the setting of  \cite[Result 1]{karoui2013robust}.  
\item Remove the smoothness and strong convexity assumptions on $\rho$. 
\item Add a regularization term to the objective function
   $\cL(\theta)$ cf. Eq.~(\ref{eq:Mestimation}), of the form $\sum_{i=1}^p
  J(\theta_i)$, with 
$J:\reals\to\reals$ a convex penalty. For $\ell_1$ penalty and
$\ell_2$ loss, this reduces to the Lasso, studied in
\cite{BayatiMontanariLASSO}.
\item Generalize the present results to non-Gaussian designs. We
  expect --for instance-- that they should hold universally across
  matrices $\bX$ with i.i.d. entries (under suitable moment
  conditions). A similar universality result was established in
  \cite{BM-Universality} for compressed sensing. 
\end{enumerate}

Let us  mention that alternative proof techniques would be 
worth exploring as well. In particular, Shcherbina and Tirozzi
\cite{shcherbina2003rigorous} define a statistical mechanics model 
with energy function that is analogous to  the loss $\cL(\theta)$,
cf. Eq.~(\ref{eq:Mestimation}), and  Talagrand \cite[Chapter
3]{TalagrandVolI} proves further results on the same model. While this
treatment focuses on estimating a certain partition function, in the
case of strongly convex $\rho$ it should be possible to extract
properties of the minimizer from a `zero-temperature' limit.

Finally, Rangan \cite{RanganGAMP} considers a similar regression model
to the one studied here using approximate message passing algorithms,
albeit from a Bayesian point of view.
%
%
\section{Duality between robust regression and regularized least squares}

The reader might have noticed many analogies between the analysis in
the last pages and earlier work on estimation in the underdetermined
regime $n<p$ using the Lasso
\cite{DMM09,DMM-NSPT-11,donoho2011compressed,BayatiMontanariLASSO}.
Most specifically, the central tool in our proof of the correctness
of State Evolution is a set of lemmas and theorems about
analysis of recursive systems that were developed to understand the Lasso.
That the same machinery directly gives results in robust regression
 \-- see for example our proof of correctness of State Evolution in
Appendix \ref{app:StateEvolution} below \-- might seem particularly unexpected.
In this section we briefly point out that the two problems are so closely linked that phenomena which appear in
one situation are bound to appear in the other.

\subsection{Duality of optimization problems}

In a very strong sense, solving an M-estimation problem
with $p < n$ is the very same thing as solving a related penalized regression
problem in $\tp > \tn$.
Given a convex function $J:\reals\to\reals$,  define the $\rho$ function
\begin{align}
\rho_J(z) \equiv \min_{x\in\reals}\Big\{\frac{1}{2}(z-x)^2+J(x)\Big\} \label{eq:RhoJdef}
\end{align}
We then have the M-Estimation problem
\beq \label{eq:MJ}
    (M_J) \qquad\qquad  \min_{\theta\in\reals^p} \;\; \sum_{i=1}^n \rho_J(Y_i - \langle X_i , \theta \rangle)
\eeq

This problem has $p < n$ and is generically a determined problem.
We now construct a corresponding underdetermined problem with the `same' solution.
Set $\tn= n-p$, $\tp=n$.
We  soon will construct a vector/matrix pair
$(\tY\in\reals^{\tn},\tbX\in\reals^{\tn \times \tp})$
obeying $\tn < \tp$, where $\tY$ and $\tbX$ are related to $Y$ and $\tbX$ in a specific way.
With this pair we pose the $J$-penalized least squares problem
\begin{align}
(L_{J}) \qquad\qquad \min_{\beta\in\reals^{\tp}} \;\;\frac{1}{2}\|\tY-\tbX\beta\|_2^2+\sum_{i=1}^{\tp}J(\beta_i)\, .
\label{eq:Regularized}
\end{align}
with solution $ \hbeta (\tY;\tbX) $, say.

Here is the specific pair that links $(M_J)$ with $(L_J)$.
We let  $\tbX$ be a matrix with orthonormal rows such that $\tbX \bX = 0$, i.e.
\begin{align}
{\rm null}(\tbX) = {\rm image}( \bX) \, ,  \label{eq:Xdual}
\end{align}
finally, we set  $\tY = \tbX Y$.

\subsubsection{The Lasso-Huber connection}

Of special interest is the case $J(x) = \lambda\,|x|$ in which case $(L_J)$ of
 (\ref{eq:Regularized}) defines the Lasso estimator.
  Then $\rho_J(x) = \Huber (x;\lambda)$
 is the Huber loss and $(M_J)$ of (\ref{eq:MJ}) defines the Huber M-estimate.
 Indeed, in that case  $(L_J)$ is more classically presented as 
\begin{align}
(\mbox{Lasso}_\lambda) \qquad \min_{\beta\in\reals^{\tp}} \frac{1}{2}\|\tY-\tbX\beta\|_2^2+ \lambda \sum_{i=1}^{\tp}| \beta_i|\, ,
\label{eq:Lasso}
\end{align}
while $(M_J)$ is more classically presented as
\beq \label{eq:MHuber}
    (\mbox{Huber}_\lambda) \qquad  \min_{\beta \in \reals^p}
    \sum_{i=1}^n  \Huber  (Y_i - \langle X_i , \beta \rangle;\lambda)
\eeq

In this special case, our general result from the next section implies the following:

\begin{propo}
With problem instances $(Y,X)$ and $(\tY,\tbX)$ related as above,
the optimal values of  the Lasso problem $(\mbox{Lasso}_\lambda)$ 
and  the Huber problem $(\mbox{Huber}_\lambda)$ are identical.
The solutions of the two problems are in one-one-relation. In
particular,  we have
\begin{align}
\htheta  =  (\bX^{\sT}\bX)^{-1}\bX^{\sT}(Y-\hbeta)\, .
\end{align}
\end{propo}

In a sense the Lasso problem solution $\hbeta$  is finding the outliers in $Y$;
once the solution is known,
the solution of the M-estimation problem is simply a
least squares regression on adjusted data $Y_{\rm adj} \equiv (Y-\hat{\beta})$ with
outliers removed.

\subsubsection{General duality result}

We will now show that the problem (\ref{eq:MJ}) is dual to
(\ref{eq:Regularized}) under or special choice of $(\tY,\tbX)$,
via (\ref{eq:Xdual}).

\emph{Notation.} For $x\in\reals^n$, we denote by $\partial \rho(x)$ the subgradient of
the convex function $\sum_{i=1}^n\rho(x_i)$, at $x$. Analogously, for 
$z\in\reals^{\tp}$, we denote by $\partial J(z)$ the subgradient of
the convex function $\sum_{i=1}^{\tp} J(z_i)$, at $z$. %
\begin{propo}
Assume that $\rho(\,\cdot\,) = \rho_J(\,\cdot\,)$, that $\tbX$ has
orthonormal rows with ${\rm null}(\tbX) = {\rm image}( \bX)$, and
finally that $\tY = \tbX\, Y$. Then the solutions of the regularized
least squares problem (\ref{eq:Regularized}) are in one-to-one
correspondence with the solutions of the robust regression problem
(\ref{eq:Mestimation}), via the mappings
\begin{align}
\hbeta & =  Y-\bX\htheta- u\, ,\;\;\;\;\;\;\;\; u\in {\rm
  null}(\bX^{\sT})\cap \partial\rho(y-\bX\htheta)\, ,\\
\htheta & =  (\bX^{\sT}\bX)^{-1}\bX^{\sT}(Y-\hbeta)\, .
\end{align}
\end{propo}

\begin{proof}
`Differentiating'
Eq.~(\ref{eq:RhoJdef}) it is easy to see that
\begin{align}
u \in\partial\rho(x) \;\;\;\mbox{ if and only if }\;\;\; u
\in \partial J(x-u)\, . \label{eq:Subdifferential}
\end{align}

First assume  $\htheta$ is a minimizer of problem
(\ref{eq:MJ}). This happens if and only if there exists $u\in\reals^n$ such that
\begin{align}
\bX^{\sT} u  =0\, ,\;\;\;\;\;\;\;\;\; u\in\partial\rho(Y-\bX\htheta)\,\label{eq:StationarityRobust}
.
\end{align}
We then claim that  $\hbeta \equiv Y-\bX\htheta- u$ is a minimizer of
Eq.~(\ref{eq:Regularized}). Indeed
\begin{align}
\tbX^{\sT}(\tY-\tbX\hbeta)&=\tbX^{\sT}\tbX (Y-\hbeta) \\
&= \tbX^{\sT}\tbX\big(\bX\htheta+u \big) = u\, ,
\end{align}
where the last identity follows since, by Eq.~(\ref{eq:Xdual}), 
${\rm null}(\bX^{\sT}) = {\rm image}( \tbX^{\sT})$, and hence $u\in {\rm image}( \tbX^{\sT})$
by Eq.~(\ref{eq:StationarityRobust}). 
Using again Eqs.~(\ref{eq:StationarityRobust}) and
(\ref{eq:Subdifferential}), we deduce that $u\in \partial J(\hbeta)$,
i.e.
\begin{align}
\tbX^{\sT}(\tY-\tbX\hbeta) \in \partial J(\hbeta)\, ,
\end{align}
which is the stationarity condition for the problem
(\ref{eq:Regularized}).

Viceversa a similar argument shows that, given 
$\hbeta$ that minimizes Eq.~(\ref{eq:Regularized}), and 
$\htheta \equiv (\bX^{\sT}\bX)^{-1}\bX^{\sT}(Y-\hbeta)$ is a minimizer 
of the robust regression problem (\ref{eq:MJ}).
\end{proof}
\subsection{Comparison to AMP in the $p > n$ case}

The last section raises the possibility that
the phenomena found in this paper for M-estimation in the $p < n$
case  are actually isomorphic to 
those found in our previous work on penalized
regression  in the $p > n$ case; \cite{DMM09,DMM-NSPT-11,donoho2011compressed,BayatiMontanariLASSO}.
Here we merely content ourselves with sketching
a few similarities.

To be definite,
 consider robust regression using the Huber loss \cite{HuberMinimax,HuberBook}
$\rho(x) = x^2/2$ for $|x|\le \lambda$ and $\rho(x) =
\lambda|x|-\lambda^2/2$ otherwise. In this case it is easy to see that
\begin{align}
\dual(z;b) = 
\begin{cases}
\lambda b & \mbox{ if $z>\lambda(1+b)$,}\\
b\, z/(1+b) & \mbox{ if $|z|\le\lambda(1+b)$,}\\
-\lambda b  & \mbox{ if $z<-\lambda(1+b)$.}
\end{cases}
\end{align}
In order to make contact with the Lasso, recall the definition of soft
thresholding operator $\eta(x;\alpha) = \sign(x)\,(|x|-\alpha)_+$. We
have the relationship
\begin{align}
\dual(z;b) = \frac{b\,z}{1+b}-\eta\Big(\frac{b\,z }{1+b};\lambda
b\Big)\, .
\end{align}
Letting $c_t\equiv b_t/(1+b_t)$, the state evolution equation (\ref{eq:StateEvolution}), then reads
\begin{align}
\tau_{t+1}^2 = \delta c_t^2 \; \E\Big\{\Big[\eta\Big(W+\tau_t\,
Z;\lambda (1+b_t)\Big)-W-\tau_t\,
Z\Big]^2\Big\}\, ,.
\end{align}
This is very close to the state evolution equation
in compressed sensing  for reconstructing
a sparse signal whose entries have distribution $F_W$, from an
underdetermined number of linear measurements; indeed in that setting we
have the state evolution recursion
\begin{align}
\tau_{t+1}^2 = \delta \; \E\Big\{\Big[\eta\Big(W+\tau_t\,
Z;\lambda \tau_t\Big)-W ]^2\Big\}\, ;
\end{align}
\cite{DMM09,DMM-NSPT-11,donoho2011compressed,BayatiMontanariLASSO}.
The connection is quite suggestive: while in compressed sensing we
look for the few non-zero coefficients in the signal, in robust
regression we try to identify the few outliers contaminating the linear relation.
A similar  duality was already pointed out in
\cite{donoho2009observed}, although in a specific setting.

%
%

\section*{Acknowledgements}

This work was
partially supported by the NSF CAREER award CCF-0743978, the NSF grant DMS-0806211, and
the grants AFOSR/DARPA FA9550-12-1-0411 and FA9550-13-1-0036.
%
%
\appendix

\section{Properties of the functions $\Prox$, $\dual$}
\label{app:Properties}

Throughout this section $\rho:\reals\to\reals$ is
convex bounded below and smooth (i.e. with bounded second derivative).
Recall the definition of
$\Prox:\reals\times\reals_{>0}\to\reals$ and $\dual:\reals\times\reals_{>0}\to \reals$,
given by
\begin{align}
\Prox(z;b) & \equiv \arg\min_{x\in\reals} \Big\{\rho(x) +
\frac{1}{2b}(x-z)^2\Big\}\, ,\\
\dual(z;b) &\equiv b\, \rho'\big(\Prox(z;b)\big)\, .
\end{align}

\begin{propo}\label{propo:Prox}
The function $\Prox:\reals\times\reals_{>0}\to\reals$ is
differentiable in its domain, with partial derivatives
\begin{align}
\frac{\partial \Prox}{\partial z}(z;b) =\left.
\frac{1}{1+b\rho''(x)}\right|_{x = \Prox(z;b)}\, ,\;\;\;\;\;
\frac{\partial \Prox}{\partial b}(z;b) =\left.
-\frac{\rho'(x)}{1+b\rho''(x)}\right|_{x = \Prox(z;b)}\, .
\end{align}
In particular, letting $\|\rho''\|_{\infty} \equiv
\sup_{x\in\reals}\rho''(x)$, and for any fixed $b$, $z\mapsto \Prox(z;b)$ is strictly
increasing and Lipschitz continuous, with 
\begin{align}
\frac{1}{1+b\|\rho''\|_{\infty}}\le\frac{\partial \Prox}{\partial z}(z;b) \le 1
\end{align}
\end{propo}
\begin{proof}
Since, for  $b>0$, $x\mapsto \rho(x) + (x-z)^2/(2b)$ is differentiable
and strongly convex, $x= \Prox(z;b)$ is uniquely determined by setting
to zero the first derivative:
\begin{align}
x+b\rho'(x) -z = 0\, .\label{eq:FirstOrderProx}
\end{align}
The claim then follows from the Implicit Function theorem.
\end{proof}

\begin{propo}\label{propo:Phi}
For $(z,b)\in\reals\times\reals_+$, we have
\begin{align}
\dual(z;b) = z-\Prox(z,b)\, ,
\end{align}
and hence $\dual$ is differentiable, with partial derivatives
\begin{align}
\frac{\partial \dual}{\partial z}(z;b) =\left.
\frac{b\rho''(x)}{1+b\rho''(x)}\right|_{x = \Prox(z;b)}\, ,\;\;\;\;\;
\frac{\partial \dual}{\partial b}(z;b) =\left.
\frac{\rho'(x)}{1+b\rho''(x)}\right|_{x = \Prox(z;b)}\, .
\end{align}
In particular, for any fixed $b$, $z\mapsto \dual(z;b)$ is strictly
increasing and Lipschitz continuous, with 
\begin{align}
\frac{b\inf_{x\in\reals}\rho''(x)}{1+b\inf_{x\in\reals}\rho''(x)}\le\frac{\partial
  \dual}{\partial z}(z;b) \le 
\frac{b\|\rho''\|_{\infty}}{1+b\|\rho''\|_{\infty}}\, . \label{eq:BracketPsiP}
\end{align}
\end{propo}
\begin{proof}
Using again the stationarity condition (\ref{eq:FirstOrderProx}) that
holds for $x=\Prox(z;b)$, we have
\begin{align}
\Prox(z;b)+b\rho'(\Prox(z;b)) -z = 0\, ,
\end{align}
which is our first claim. The other claims immediately follow by calculus.
\end{proof}

Finally, we prove that Eq.~(\ref{eq:BtDef}) that defines $b_t$ as a
function of $\tau_t$ always has at least one solution.
\begin{lemma}\label{lemma:BtDef}
For $\tau>0$ fixed, let $G:\reals_{>0}\to\reals$ be defined by
\begin{align}
G(b) \equiv \E\Big\{\dual'(W+\tau\,Z;b)\Big\}\, .
\end{align}
Then for any $a\in (0,1)$, the set of solutions 
\begin{align}
{\cal S}_a \equiv\big\{b\in\reals_{>0}:\;  G(b)= a\big\}\, ,
\end{align}
is closed and non-empty.
\end{lemma}
\begin{proof}
It follows immediately from the continuity properties of $\dual$ that
$b\mapsto G(b)$ is continuous. The claim follows by proving that 
$\lim_{b\to 0}G(b) = 0$ and $\lim_{b\to \infty}G(b) = 1$. 

By  Proposition \ref{propo:Phi} equation (\ref{eq:BracketPsiP})
 $0\le \dual'(z;b)\le 1$.
The limit $b\to 0$ follows from dominated convergence since, by 
the upper bound in (\ref{eq:BracketPsiP})
 $\lim_{b\to 0}\dual'(z;b) = 0$ for each $z$. 

In order to  obtain the limit as $b\to\infty$, note that by Stein Lemma:
\begin{align}
G(b) =\frac{1}{\tau}\E\Big\{Z\, \dual(W+\tau\,Z;b)\Big\}\, .
\end{align}
Since $0\le
\dual'(z,b)\le 1$, the integrand is  bounded in modulus by an
integrable quantity. We can therefore use again dominated
convergence. Now $\lim_{b\to\infty} \Prox(z;b) =
\arg\min_{x\in\reals}\rho(x)\equiv c_0$ and hence $\lim_{b\to\infty} \dual(z;b) =
z-c_0$. By dominated convergence we obtain  
\begin{align}
\lim_{b\to\infty}G(b) =\frac{1}{\tau}\E\Big\{Z\, (W+\tau Z-c_0)\Big\} = 1\, .
\end{align}
\end{proof}
%
%
%
%
%
%
%
%
\section{Proof of correctness of State Evolution (Theorem \ref{thm:StateEvolution})}
\label{app:StateEvolution}

We will show correctness of State Evolution for
the AMP algorithm using analytically defined $b_t$.
Namely, we suppose that 
with $b_t$ defined recursively as the smallest positive solution of 
the second equation in this system:
\begin{align}
\tau_{t+1}^2 &= \delta \; \E\Big\{\dual(W+\tau_t\, Z;b_t)^2\Big\}\, ,\\
\frac{1}{\delta} &= \E\Big\{\dual'(W+\tau_t\,Z;b_t)\Big\}\, .
\end{align}

For analysis purposes, we consider   a recursion equivalent to the AMP
recursion,  in which the data are recentered
and the recursion is recast around recentered variables.
We change the initial condition of the AMP
iteration by letting $\htheta^{{\rm cen},0} = \htheta^0-\theta_0$, and change
data by letting $Y^{{\rm cen}} = Y-\bX\theta_0\equiv W$. Applying the  AMP
recursion in these new coordinates gives the new
trajectory $\htheta^{{\rm cen},t} = \htheta^t-\theta_0$
for all $t$, and $\res^{{\rm cen},t} = \res^t$ for all $t$. 

The new trajectory follows the recursion
\begin{align}
\res^{{\rm cen},t} & = W-\bX\htheta^{{\rm cen},t}+\dual(\res^{{\rm cen},t-1};b_{t-1})\, ,\label{eq:AMP1again}\\
\htheta^{{\rm cen},t+1} & = \htheta^{{\rm cen},t}+\delta\bX^{\sT}\dual(\res^{{\rm cen},t};b_t)\, , \label{eq:AMP2again}
\end{align}

In this form, the recursion can be reduced to a recursion studied in \cite{BM-MPCS-2011},
for which State Evolution has been proven correct.
The reduction is to introduce a recursion generating iterates
 $\{\vt^t,S^t\}$ that approximates closely the iterates
$\{\htheta^{{\rm cen},t},\res^{{\rm cen},t}\}$ defined by (\ref{eq:AMP1again}),(\ref{eq:AMP2again}). The new sequence is defined by letting $\vt^0
= \htheta^0-\theta_0$ and, for all $t \ge 0$
\begin{align}
S^t & = -\bX\vt^t+\dual(W+S^{t-1};b_{t-1})\, , \label{eq:Simplified1}\\
\vt^{t+1} & = \delta\bX^{\sT}\dual(W+S^t;b_t)+q_t\vt^t\, ,\label{eq:Simplified2}
\end{align}
where 
\begin{align}
q_t = \delta\Big\{\frac{1}{n}\sum_{i=1}^n\dual'(W_i+S^t_i;b_t)\Big\}\, . 
\end{align}
The only difference between this recursion and the previous one
cf. Eqs.~(\ref{eq:AMP1again}), (\ref{eq:AMP2again}), lies in
the new coefficient $q_t$, which was identically equal to $1$ in the previous recursion.
The benefit of this specific recursion is that we already know that State Evolution is correct.
\begin{lemma}\label{lemma:ApplyingSE}
Under the assumptions of Theorem \ref{thm:StateEvolution}, we have,
for any fixed $t\ge 0$, 
\begin{align}
\lim_{n\to\infty} \frac{1}{p}\sum_{i=1}^p\test(\vt^t_i)=_{{\rm a.s.}} \, &
\E\{\test(\sqrt{\delta}\, \tau_t\, Z)\}\,    \label{eq:SE1}\\
\lim_{n\to\infty}\frac{1}{n}\sum_{i=1}^n\test_2(S^t_i,W_i)
=_{{\rm a.s.}} \, & \E\Big\{\test_2(\tau_t\, Z,W)\Big\}\,. \label{eq:SE2} 
\end{align}
\end{lemma}
\begin{proof}
This is an immediate application of Theorem 2 in \cite{BM-MPCS-2011}.
That Theorem considers general recursions which include (\ref{eq:Simplified1})-(\ref{eq:Simplified2})
as a special case,
and  corresponding state evolution equations, and shows the correctness of state
evolution, in the process establishing conclusions of the precise form shown in the conclusion
of this lemma. So it is simply a matter of establishing the correspondence of variables.

In the original notation of \cite{BM-MPCS-2011}, the generalized AMP recursions  studied are
\begin{align}
b^{t}      & =  A    q^t    - \lambda_t m^{t-1}  \, \label{eq:Invoke1} \\
h^{t+1} & =  A^* m^t - \xi_t q^t \,\label{eq:Invoke2}
\end{align}
where $b^{t}$, $h^t$, $q^t$ and $m^t$ are vectors and $\lambda_t$ and $\xi_t$
scalars. In addition, the vectors $q^t = f_t(h^t)$ and $m^t =g^t(b^t,w)$ are produced by element wise
applications of nonlinearities $f_t$ and $g_t$, the latter involving the random vector $w$.
Here $A$ is a rectangular $n \times N$ random matrix with iid Gaussian entries.
The scalars $\xi_t = \langle g'_t(b^t,w) \rangle$ and $\lambda_t = \frac{1}{\delta} \langle f'_t(h^t) \rangle$,
where $\langle \cdot \rangle$ denotes an empirical mean over the entries in a vector.
and the iteration takes $m^{-1}=0$. For state evolution, 
the Theorem 2 assumes the sequence of initial conditions $q^0$ obeys
\[
  \overline{\sigma}^2_0 = \lim_{N \goto \infty} \frac{1}{N \delta} = \| q^0 \|_2^2.
\] 
and the state evolution recursion involves the pair of variables
\[
      \overline{\tau}_t^2 = \E \{ g_t^2(\overline{\sigma}_t Z,W) \}, \qquad   \overline{\sigma}_t^2 = \E \{ f_t^2(\overline{\tau}_{t-1} Z) \}. 
\]
Table \ref{table:corrtable} sets up a  `dictionary' of correspondences between this
paper and \cite{BM-MPCS-2011}. 

\begin{table}[h!]
\begin{center}
\begin{tabular}{| c| c | c | c | c | c | c |}
\hline
(\ref{eq:Simplified1})                         & $\vt^{t+1}  = \delta\bX^{\sT}\dual(W+S^t;b_t)+q_t\vt^t$ & $\vt^{t+1}$ & $\bX^{\sT}$&  $\delta \dual(W+S^t;b_t)$ & $q_t$ & $\vt^t$ \\
(\ref{eq:Invoke1})& $h^{t+1}  =  A^* m^t - \xi_t q^t $ & $h^{t+1}$ & $A^*$ &  $m^t$  & $\xi_t$  & $- q^t$ \\ 
\hline
(\ref{eq:Simplified2})                         & $S^t = -\bX\vt^t+\dual(W+S^{t-1};b_{t-1})$ & $S^t $ & $\bX$ & $-\vt^t $ & 1 & $ \dual(W+S^{t-1};  b_{t-1})$ \\
(\ref{eq:Invoke2})& $ b^{t}      =  A    q^t    - \lambda_t m^{t-1}$ & $b^{t}  $ & $  A $  & $q^t$ & $- \lambda_t$ & $  m^{t-1}$ \\
\hline
\end{tabular}
\caption{Correspondences between terms in the recursions of this paper, (\ref{eq:Simplified1})-(\ref{eq:Simplified2}),
and the recursions (\ref{eq:Invoke1})-(\ref{eq:Invoke2}), analyzed in \cite{BM-MPCS-2011}.}
\label{table:corrtable}
\end{center}
\end{table}
We get exact correspondence between the two systems, provided we identify 
$\delta \dual(W+S^t;b_t)$  with $m^t = g_t(b^t ; w)$
and $-\delta h^t$ with $f_t(h^t)$. One has, in particular, that  $\lambda_t = \frac{1}{\delta} \langle f'_t(h^t) \rangle = - 1$,
and that $\xi_t = \langle g'_t(b^t,w) \rangle = \langle \delta \dual'(W+S^t;b_t) \rangle = q_t$.

In the \cite{BM-MPCS-2011} general study of state evolution, there are two state variables $\overline{\tau}_t$ and $\overline{\sigma}_t$.
However, when applied here, the distinction vanishes.
The  variable called $\overline{\tau}_t^2$ corresponds with $  \delta^2 \E \{ \Psi(W + \overline{\sigma}_t Z)^2\}$
while the variable $\overline{\sigma}_t^2$  corresponds with $\E (\vt^{t})^2$. Using facts about $\tau_t^2$ in this paper, we have the
identities $\overline{\sigma}_t^2 = \delta \tau_t^2$ and $\overline{\tau}_t^2 = \delta \tau_t^2$.
Equations (\ref{eq:SE1})-(\ref{eq:SE2}) now follow from Theorem 2 of \cite{BM-MPCS-2011}.
\end{proof}

Theorem \ref{thm:StateEvolution} now follows
from  the equivalence of the last two recursions -- i.e. equivalence of (\ref{eq:AMP1again})-(\ref{eq:AMP2again}) with
(\ref{eq:Simplified1})-(\ref{eq:Simplified2}).
\begin{lemma}\label{lemma:SimplifyingIteration}
Under the assumptions of Theorem \ref{thm:StateEvolution}, we have,
for any fixed $t\ge 0$,
\begin{align}
\lim_{n\to\infty} \frac{1}{p}\|\htheta^{{\rm cen},t}-\vt^t\|_2^2
=_{{\rm a.s}} \, 0\, ,
\;\;\;\;\;\;\;\;\;\; \lim_{n\to\infty} \frac{1}{n}\|\res^{{\rm
    cen},t}-S^t-W\|_2^2 =_{{\rm a.s.}} \, 0\, .
\end{align}
\end{lemma}
\subsection{Proof of Lemma \ref{lemma:SimplifyingIteration} (Equivalence of recursions)}

Throughout this proof, we will drop the superscript `cen' from $\res^{{\rm
    cen},t}$ and $\htheta^{{\rm cen},t}$. 
Define $S^t_+\equiv W+S^t$, whence
\begin{align}
S^t_+ & = W-\bX\vt^t+\dual(S_+^{t-1};b_{t-1})\, ,\\
\vt^{t+1} & = \delta\bX^{\sT}\dual(S^t_+;b_t)+q_t\vt^t\, ,\label{eq:VTeqProof}
\end{align}
Comparing the first of these equations with Eq.~(\ref{eq:AMP1again}),
and using triangular inequality, we get
\begin{align}
\|\res^{t}-S^t_+\|_2& \le \|\bX\|_2\|\htheta^t-\vt^t\|_2 +
\|\dual(\res^{t-1};b_{t-1}) - \dual(S_+^{t-1};b_{t-1}) \|_2\, \\
& \le \|\bX\|_2\|\htheta^t-\vt^t\|_2 +
\|\res^{t-1}- S_+^{t-1}\|_2\, ,\label{eq:BoundDifference1}
\end{align}
where the last inequality follows since
$\dual(\,\cdot\,;b):\reals\to\reals$  is Lipschitz continuous with
Lipschitz constant at most $1$, cf Proposition \ref{propo:Phi}.

Comparing analogously  Eq.~(\ref{eq:AMP2again}) and (\ref{eq:VTeqProof}), we obtain
\begin{align}
\|\htheta^{t+1}-\vt^{t+1}\|_2& \le \delta\|\bX\|_2\,
\|\dual(\res^t;b_t)
-\dual(S^t_+;b_t)\|_2+\|\htheta^t-\vt^t\|_2+|q_t-1|\,\|\vt^t\|_2\\
&\le \delta\|\bX\|_2\,
\|\res^t-S^t_+\|_2+\|\htheta^t-\vt^t\|_2+|q_t-1|\,\|\vt^t\|_2\, . \label{eq:BoundDifference2}
\end{align}
Iterating the upper bounds (\ref{eq:BoundDifference1}),
(\ref{eq:BoundDifference2}), and using the fact that $\vt^0=\htheta^0$,
we conclude that there exists a 
constant $A=A(\delta)<\infty$ such that
\begin{align}
\|\htheta^{t}-\vt^{t}\|_2\le (A\|\bX\|_2)^{2t}\sum_{\ell=0}^{t-1}|q_{\ell}-1|\,\|\vt^{\ell}\|_2\label{eq:Gromwall}
\end{align}
By Lemma \ref{lemma:ApplyingSE}, we have, almost surely
\begin{align}
\lim_{n\to\infty}\frac{1}{\sqrt{p}}\|\vt^{\ell}\|_2 =
\tau_{\ell}<\infty\, ,
\end{align}
and
\begin{align}
\lim_{n\to\infty}q_{t}&=
\delta\lim_{n\to\infty}\frac{1}{n}\sum_{i=1}^n\dual'(W_i+S^t_i;b_t)\\
&=
\delta\E\{\dual'(W+\tau_t Z;b_t) \}= 1\, , 
\end{align}
where the second identity follows from
Lemma \ref{lemma:ApplyingSE}
and, in 
the third, we used the definition of $b_t$.
(Note that we are applying here Lemma \ref{lemma:ApplyingSE} to
$\test(\,\cdot\,)=\dual'(\,\cdot\,;b_t)$ which is bounded and
non-negative but not necessarily continuous. However, since $W+\tau_t
Z$ has a density for every $\tau_t>0$, the limit holds by a standard
weak convergence argument, approximating $\test$ by simple functions.
Namely, we construct a sequence of simple functions $\xi_\ell$ such
that $\xi_{\ell}(t)\le \xi(t)\le \xi_{\ell}(t) +(1/\ell)$ for all $t$,
and apply Lemma \ref{lemma:ApplyingSE} --which implies weak convergence
of the empirical distribution of $\{W_i+S^t_i\}$-- to $\xi_{\ell}$.)

Finally, it is a standard result in random matrix theory
\cite{Guionnet} that $\lim_{n\to\infty}\|\bX\|_2 = C(\delta)<\infty$.
Hence, by taking the limit of Eq.~(\ref{eq:Gromwall}) we get, almost
surely,
\begin{align}
\lim_{n\to\infty}\frac{1}{\sqrt{p}}\|\htheta^{t}-\vt^{t}\|_2 = 0\, .
\end{align}
The norm $\|\res^{t}-S^t_+\|_2$ is then controlled using Eq.~(\ref{eq:BoundDifference1}). 

%
%
\section{Proof that AMP converges to the M-estimator (Theorem \ref{thm:Convergence})}

Notice first of all that, by construction, $\tau_t^2=\tau_*^2$, $b_t =
b_*$ for all $t$.

Given $\delta$, $\rho$ as in the statement of the theorem and
$\tau_*$, $b_*$ a solution of the fixed point equation
(\ref{eq:FixedPoint1}), (\ref{eq:FixedPoint2}), we define the doubly
infinite matrix
$\Corr = (\Corr_{t,s})_{t,s\ge 0}$ by letting, recursively for $t,s\ge 0$
\begin{align}
\Corr_{t+1,s+1} = \delta\E\{\dual(W+Z_t;b_*) \dual(W+Z_s;b_*)\}\, ,
\end{align}
where  the expectation is with respect to $(Z_t,Z_s)$ jointly
Gaussian, with zero means and covariance $\E\{Z_t^2\} = \Corr_{t,t}$,
$\E\{Z_s^2\} = \Corr_{s,s}$, $\E\{Z_tZ_s\} = \Corr_{t,s}$, independent
of $W\sim F_W$.
This is supplemented with the boundary condition $\Corr_{0,0} =
\tau_*^2$ and $\Corr_{0,t} = \Corr_{t,0}= 0$ for $t>0$.

Notice that, in particular, $\Corr_{s,t}= \Corr_{t,s}$ for all $s,t\ge
0$ and $\Corr_{t,t}=\tau_*^2$ for all $t$.

The significance of these quantities is clarified by the following 
result.
\begin{lemma}\label{lemma:TwoTimesSE}
Under the hypotheses of Theorem \ref{thm:StateEvolution}, further
assume that $\tau_*^2$ and $\Corr$ are defined as above.
Then, for any $t,s\ge 0$,
\begin{align}
\lim_{n\to\infty}\frac{1}{p}\sum_{i=1}^p\test_2(\htheta^t_{i}-\theta_{0,i},\htheta_i^s-\theta_{0,i})
& =_{{\rm a.s.}} \, \E\test_2(\sqrt{\delta}\, Z_t,\sqrt{\delta}\, Z_s)\, ,\\
\lim_{n\to\infty}\frac{1}{n}\sum_{i=1}^n\test_2(\res^t_{i}-W_i,\res_i^s-W_i)
& =_{{\rm a.s.}} \, \E\test_2( Z_t, Z_s)\, ,
\end{align}
where  the expectation is with respect to $(Z_t,Z_s)$ jointly
Gaussian, with zero means and covariance $\E\{Z_t^2\} = \Corr_{t,t}$,
$\E\{Z_s^2\} = \Corr_{s,s}$, $\E\{Z_tZ_s\} = \Corr_{t,s}$, independent
of $W\sim F_W$.
\end{lemma} 
The proof is deferred to Section \ref{sec:TwoTimesSE}.

As a special case of the latter result, we have
\begin{align}
\lim_{n\to\infty}\frac{1}{p}\|\htheta^{t}-\htheta^s\|_2^2 &\, =_{a.s.} \,
2\delta\,\big(\tau_*^2-\Corr_{t,s}\big)\, ,\label{eq:DiffTheta}\\
\lim_{n\to\infty}\frac{1}{n}\| \res^{t}-\res^s\|_2^2 &\, =_{a.s.} \,
2\,\big(\tau_*^2-\Corr_{t,s}\big)\, .\label{eq:DiffR}
\end{align}

The following lemma provides information about the asymptotic behavior
of $\Corr_{t,s}$.  Its proof is deferred to Section \ref{sec:ProofCorrBehavior}.
\begin{lemma}\label{lemma:CorrBehavior}
Let $\tau_*$, $\Corr$ be defined as above for $\delta>1$. Then
\begin{align}
\lim_{t\to\infty} \Corr_{t,t+1} = \tau_*^2
\end{align}
\end{lemma}

Applying this result to Eqs.~(\ref{eq:DiffTheta}) and (\ref{eq:DiffR})
we get,
for any fixed $h\in\naturals$,
\begin{align}
\lim_{t\to\infty}\lim_{n\to\infty}\frac{1}{p}\|\htheta^{t+h}-\htheta^t\|_2^2
&\, =_{a.s.} \, 0\, ,\label{eq:ConvergenceT}\\
\lim_{t\to\infty}\lim_{n\to\infty}\frac{1}{p}\|\res^{t+h}-\res^t\|_2^2
&\, =_{a.s.} \, 0\, .\label{eq:ConvergenceR}
\end{align}
(The case $h>1$ follows from $h=1$ by the triangle inequality.)

We are now ready to prove Theorem \ref{thm:Convergence}. 
Recall that $\cL(\theta) = \cL(\theta;Y,\bX)$ denotes the loss
function defined in Eq.~(\ref{eq:Mestimation}), and that its gradient
and Hessian are given by
\begin{align}
\nabla_{\theta}\cL(\theta) & =
-\sum_{i=1}^n\rho'(Y_i-\<X_i,\theta\>)\, X_i\, ,\\
\nabla^2_{\theta}\cL(\theta) & =
\sum_{i=1}^n\rho''(Y_i-\<X_i,\theta\>)\, X_iX_i^{\sT}\, .
\end{align}
In particular, letting $\sigma_{\rm min}(\bX)$ denote the minimum
non-zero singular value of $\bX$, we have
\begin{align}
\lambda_{\rm min} (\nabla^2_{\theta}\cL(\theta)) \ge \inf_{x\in
  \reals}\rho''(x)\, \cdot\sigma_{\rm min}(\bX)^2\, .
\end{align}
Using the hypothesis of strong convexity and standard concentration of
measure for the singular values of Wishart matrices
\cite{Vershynin-CS}, these exists constants $c_0,c_1, n_0>0$ for $\delta>1$
such that for any $n\ge n_0$, 
\begin{align}
\prob\Big(\nabla^2_{\theta}\cL(\theta)\succeq c_0\, \id\;\;\;\; \forall
\theta\in\reals^p\Big) \ge \, 1-e^{-c_1\, n}\, .
\end{align}

As a consequence, with probability at least $1-e^{-c_1n}$, we have 
\begin{align}
\cL(\htheta^t) \ge \cL(\htheta)\ge \cL(\htheta^t)
+\<\nabla_{\theta}\cL(\htheta^t),\htheta-\htheta^t\>
+\frac{1}{2}\, c_0\, \|\htheta-\htheta^t\|_2^2\, . 
\end{align}
Hence using Cauchy-Schwartz
\begin{align}
\|\htheta-\htheta^t\|_2\le \frac{2}{c_0}\,
\|\nabla_{\theta}\cL(\htheta^t)\|_2\, .
\end{align}
The last step of the proof consists in showing  that, almost surely
\begin{align}
\lim_{t\to\infty}\lim_{n\to\infty}\frac{1}{p}\|\nabla_{\theta}\cL(\htheta^t)\|_2^2 =
0\, .\label{eq:GradToZero} 
\end{align}

In order to prove this claim, reconsider Eq.~(\ref{eq:AMP1}), for time
$t+1$, with
$b_t=b_*$. Using the fact that $\dual(z;b_*) = z-\Prox(z;b_*)$, this
can be rewritten as
\begin{align}
\Prox(\res^{t};b_*) = Y-\bX\htheta^{t+1}+\res^t-\res^{t+1}\, .\label{eq:ProxEndProof}
\end{align}
By Eq.~(\ref{eq:AMP2}), and recalling that $\dual(z;b) =
b\,\rho'(\Prox(z;b))$,  we have
\begin{align}
\frac{1}{b_*\delta}\,
\big(\htheta^{t+1}-\htheta^t\big)&= \bX^{\sT} \rho'(\Prox(\res^t;b_*)) \\
&=\bX^{\sT} \rho'\Big( Y-\bX\htheta^{t+1}+\res^t-\res^{t+1}\Big)  \, ,
\end{align}
where the last identity followed by Eq.~(\ref{eq:ProxEndProof}). Using
the triangle inequality and noting that, by the smoothness assumption 
$C\equiv \sup_{z\in\reals}\rho''(z)<\infty$, we get
\begin{align}
\|\bX^{\sT} \rho'\Big( Y-\bX\htheta^{t+1})\|_2
\le \frac{1}{b_*\delta}\,
\|\htheta^{t+1}-\htheta^t\|_2+ C\|\bX\|_2\|\res^t-\res^{t+1}\|_2\, .
\end{align}
Hence, using Eqs~(\ref{eq:ConvergenceT}) and (\ref{eq:ConvergenceT}),
and recalling that $\lim_{n\to\infty}\|\bX\|_2<\infty$ almost surely \cite{Guionnet},
we get
\begin{align}
\lim_{t\to\infty}\lim_{n\to\infty}\frac{1}{p}\|\bX^{\sT} \rho'\Big(
Y-\bX\htheta^{t+1})\|_2^2 = 0\, .
\end{align}
This is equivalent to the claim (\ref{eq:GradToZero}) since
$\nabla_{\theta}\cL(\theta) = -\bX\rho'(Y-\bX\theta)$.

\subsection{Proof of Lemma \ref{lemma:TwoTimesSE}}
\label{sec:TwoTimesSE}

First of all note that, due to Lemma \ref{lemma:SimplifyingIteration},
it is sufficient to prove that
\begin{align}
\lim_{n\to\infty}\frac{1}{p}\sum_{i=1}^p\test_2(\vt^t_{i},\vt_i^s)
& =_{a.s.} \, \E\test_2(\sqrt{\delta}\, Z_t,\sqrt{\delta}\, Z_s)\, ,\\
\lim_{n\to\infty}\frac{1}{n}\sum_{i=1}^n\test_2(s^t_{i},S_i^s)
& =_{a.s.} \, \E\test_2( Z_t, Z_s)\, .
\end{align}
Note that a similar statement is proved in \cite[Theorem
4.2]{BayatiMontanariLASSO} for characterizing the Lasso estimator. 
While the same argument can be followed here, we outline an
alternative argument that is based on a reduction to the setting of
\cite{JM-StateEvolution}.

We fix an even number $q\in \naturals$, and will prove the claim for
all $t,s\le T\equiv (q/2)-1$. 
Let $N \equiv n+p$.
For $t\in \{ 0,\dots, T\}$, we introduce a vector $z^t\in(\reals^q)^N$,
which we think of as a vector with entries in $\reals^q$:
$z^{t}=(\vz^t_1,\dots,\vz^t_N)$ $\vz^t_i\in\reals^q$. Its entries are defined as
follows:
\begin{align}
\vz^t_i &=
(S^0_i,0,S^{1}_i,0,S^{2}_i,0,\dots\;,S^t_i,0,0,0,\dots, 0)  &\mbox{if
  $1\le i\le n$,}\\
\vz^{t+1}_i &=
(0,\vt^1_j,0,\vt^{2}_j,0,\vt^{3}_j,\dots,0,\vt^{t+1}_i,0,0,0,\dots, 0) & \mbox{if $n+1\le i=j+n\le n+p$.}
\end{align}
Further, we let 
$A\in\reals^{N\times N}$ be a symmetric matrix with $A_{ii} = 0$,
$A_{ij} = \sqrt{n/N}\, X_{i,j-n}$ for $1\le i\le n$ and $n+1\le j\le
n+p$, and all the other entries $A_{ij}$ $i<j$
i.i.d. $\normal(0,1/N)$.
It is then easy to see that the iteration in
Eqs.~(\ref{eq:Simplified1}), (\ref{eq:Simplified2}) is equivalent to
the following
\begin{align}
z^{t+1} = A\, f(z^t;t) -\sB_t\, f(z^{t-1};t-1)\, . \label{eq:MultiTimeAMP}
\end{align}
Here, for each $t$, $f(\,\cdot\,;t):(\reals^q)^N\to(\reals^q)^N$ is
separable in the following sense
\begin{align}
f(z;t) = (f_1(\vz_1;t);f_2(\vz_2;t);\dots;f_N(\vz_N;t))\, ,
\end{align}
with $f_i(\,\cdot\,;t):\reals^q\to\reals^q$.
These are defined as follows (letting $\dual_{t,i}(x) = \dual(W_i+x;b_t)$ and
$h = \sqrt{(1+\delta)/\delta}$)
\begin{align}
f_i(\vz_i;t) &=
(0,\delta h\,\dual_{0,i}(\vz_{i,1}),0,\delta h\,\dual_{1,i}(\vz_{i,3}),0,
\dots,0,\delta h\,\dual_{t,i}(\vz_{i,2t-1}),0,0,0,\dots,0) & \mbox{if  $1\le i\le n$,}\\
f_i(\vz_i;t) &=
(0,0,-h\,\vz_{i,2},0,-h\,\vz_{i,4},0,
\dots,- h\,\vz_{i,2t},0,0,0,0,\dots,0) &  \mbox{if  $n+1\le i\le n+p$.}
\end{align}
The matrix multiplication in Eq.~(\ref{eq:MultiTimeAMP}) operates in
the natural way over $(\reals^{q})^N$, namely we identified $A$ with
the Kronecker product $A\otimes \id_{q\times q}$. Explicitly,
Eq.~(\ref{eq:MultiTimeAMP}) reads
\begin{align}
\vz^{t+1}_i = \sum_{j\in [N]}A_{ij}\, f_j(\vz_j^t;t) -\sB_t\,
f_i(\vz_i^{t-1};t-1)\, . 
\end{align}
Finally $\sB_t\in\reals^{q\times q}$ is given by
\begin{align}
\sB_t = \frac{1}{N}\sum_{i=1}^N\frac{\partial f_i}{\partial
  \vz}(\vz^t_i;t)\, .
\end{align}
 
The recursion (\ref{eq:MultiTimeAMP}) is characterized in
\cite[Theorem 1]{JM-StateEvolution}, which establishes --for instance--
that, for $\test:\reals^q\to\reals$ pseudo-Lipschitz, we have,
almost surely,
\begin{align}
\lim_{N\to\infty}\frac{1}{p}\sum_{i=n+1}^{n+p} \test(\vz_i^t) =
\E\{\test({\mathbf Z}_t)\}\, .
\end{align}
Here ${\mathbf Z}_t$ is a Gaussian random vector whose covariance is
fully specified in \cite{JM-StateEvolution}. The proof of the lemma is
finished by comparing the
expressions in \cite{JM-StateEvolution} for the covariance wit the
ones in the statement of the lemma.

\subsection{Proof of Lemma \ref{lemma:CorrBehavior}}
\label{sec:ProofCorrBehavior}

First of all we introduce the notation $q_t \equiv
\Corr_{t,t+1}/\tau_*^2$.
We then have the recursion
\begin{align}
q_{t+1} &= \sH(q_t)\, ,\\
\sH(q) & = \frac{\delta}{\tau_*^2}\, \E_q\{\dual(W+\tau_*\, Z_1;b_*) \dual(W+\tau_*\,Z_2;b_*)\}\, ,
\end{align}
where expectation $\E_q$ is with respect to the centered Gaussian vector
$(Z_1,Z_2)$ with $\E_q\{Z_1^2\}=\E_q\{Z_2^2\} = 1$ and
$\E_q\{Z_1Z_2\}= q$, independent of $W\sim F_W$.
We claim that:
\begin{enumerate}
\item[$(i)$] $\sH(1) = 1$;
\item[$(ii)$] $\sH(q)$ is increasing for $q\in[0,1]$;
\item[$(iii)$] $\sH(q)$ is strictly convex for $q\in [0,1]$.
\end{enumerate} 
In order to prove $(i)$, note that, for $q=1$, $Z_1=Z_2\equiv Z\sim
\normal(0,1)$ and hence
\begin{align}
\sH(1)  = \frac{\delta}{\tau_*^2}\, \E_q\{\dual(W+\tau_*\,
Z;b_*)^2\}\, ,
\end{align}
which is equal to $1$ since $b_*$, $\tau_*$ satisfy
Eq.~(\ref{eq:FixedPoint1}). 

In order to prove $(ii)$, $(iii)$, define
\begin{align}
h_W(z) & \equiv \dual(W+\tau_*\, z ;  b_*)\, ,\\
\cH(q) & \equiv \E_q\{h_W(Z_1) h_W(Z_2)|W\}\, ,
\end{align}
We will prove that $\cH$ is strictly increasing and convex for any
$W$, whence claims $(ii)$ and $(iii)$ follow by linearity. The
argument is the same as in \cite[Lemma C.1]{BayatiMontanariLASSO}
Let $\{X_t\}_{t\ge 0}$ be the stationary Ornstein--Uhlenbeck process with
covariance $\E(X_0X_t) = e^{-t}$, and
denote by $\sE$ expectation with respect to $X$. Then 
\begin{align}
\cH(q) = \sE\{h_W(X_0) h_W(X_t)\}\Big|_{t= \log(1/q)}\, ,
\end{align}
Then we have the spectral representation (for $t= \log(1/q)$)
\begin{align}
\cH(q) = \sum_{\ell = 0}^{\infty}c_{\ell}^2\, e^{-\ell\, t} =
\sum_{\ell = 0}^{\infty}c_{\ell}^2\, q^{\ell}\, ,
\end{align}
whence the claim follows since $c_{\ell}\neq 0$ for some $\ell\ge 2$
as long as $h_W(x)$ is non-linear. 

Because of the remarks $(i)$-$(iii)$ just proven, it follows that
$\lim_{t\to\infty}q_t= 1$ (and hence $\lim_{t\to\infty}\Corr_{t,t+1} =
\tau_*^2$) if and only if $\sH'(1)\le 1$. A simple calculation yields
\begin{align}
\sH'(1) = \delta\, \E\big\{\dual'(W+\tau_*\,
Z;b_*)^2\big\}\, ,
\end{align}
where $Z\sim\normal(0,1)$. Recalling that
 $\dual'(z;b)\in (0,1)$, we have $(\dual')^2 \leq \dual'$
and so
\begin{align}
\sH'(1) \le \delta\, \E\big\{\dual'(W+\tau_*\,
Z;b_*)\big\}=1\, ,
\end{align}
where the last identity follows because $(\tau_*,b_*)$ solve
Eq.~(\ref{eq:FixedPoint2}). This finishes the proof.

%
\bibliographystyle{amsalpha}
\providecommand{\bysame}{\leavevmode\hbox to3em{\hrulefill}\thinspace}
\providecommand{\MR}{\relax\ifhmode\unskip\space\fi MR }
\providecommand{\MRhref}[2]{%
  \href{http://www.ams.org/mathscinet-getitem?mr=#1}{#2}
}
\providecommand{\href}[2]{#2}

\end{document}